\documentclass[12pt, reqno]{amsart}
\usepackage{amsmath, amsthm, amscd, amsfonts, amssymb, graphicx, color, mathrsfs}
\usepackage{pstricks,pst-node}
\usepackage[bookmarksnumbered, colorlinks, plainpages]{hyperref}
\usepackage[all]{xy} 
\usepackage{slashed}

\newtheorem{theorem}{Theorem}[section]
\newtheorem{lemma}[theorem]{Lemma}

\newtheorem{proposition}[theorem]{Proposition}
\newtheorem{corollary}[theorem]{Corollary}

\theoremstyle{definition}
\newtheorem{definition}[theorem]{Definition}
\newtheorem{example}[theorem]{Example}
\newtheorem{openproblem}[theorem]{Open problem}

\theoremstyle{remark}
\newtheorem{remark}[theorem]{Remark}
\numberwithin{equation}{section}

\begin{document}
\setcounter{page}{1}

\title[Nuclearity and Schatten classes of Fourier integral operators]{ Nuclearity, Schatten-von Neumann classes, distribution of eigenvalues and $L^p$-$L^q$-boundedness of Fourier integral operators\\ on compact manifolds}

\author[D. Cardona]{Duv\'an Cardona}
\address{
  Duv\'an Cardona:
  \endgraf
  Department of Mathematics: Analysis, Logic and Discrete Mathematics
  \endgraf
  Ghent University, Belgium
  \endgraf
  {\it E-mail address} {\rm duvanc306@gmail.com, duvan.cardonasanchez@ugent.be}
  }
  
  \author[J. Delgado]{Julio Delgado}
\address{
  Julio Delgado:
  \endgraf
  Departamento de Matem\'aticas
  \endgraf
  Universidad del Valle
  \endgraf
  Cali-Colombia
    \endgraf
    {\it E-mail address} {\rm delgado.julio@correounivalle.edu.co}
  }

\author[M. Ruzhansky]{Michael Ruzhansky}
\address{
  Michael Ruzhansky:
  \endgraf
  Department of Mathematics: Analysis, Logic and Discrete Mathematics
  \endgraf
  Ghent University, Belgium
  \endgraf
 and
  \endgraf
  School of Mathematical Sciences
  \endgraf
  Queen Mary University of London
  \endgraf
  United Kingdom
  \endgraf
  {\it E-mail address} {\rm michael.ruzhansky@ugent.be, m.ruzhansky@qmul.ac.uk}
  }

\thanks{The authors are supported  by the FWO  Odysseus  1  grant  G.0H94.18N:  Analysis  and  Partial Differential Equations and by the Methusalem programme of the Ghent University Special Research Fund (BOF)
(Grant number 01M01021). Julio Delgado is also supported by Vic. Investigaciones Universidad del Valle CI-71352. Duv\'an Cardona has been supported by the FWO Fellowship
grant No 1204824N.  Michael Ruzhansky is also supported  by EPSRC grant 
 and EP/V005529/1.
}


\begin{abstract}  We link  Sogge's type $L^p$-estimates for eigenfunctions of the Laplacian on compact manifolds with the problem of providing criteria for the $r$-nuclearity of Fourier integral operators. The classes of Fourier integral operators $I^\mu_{\rho,1-\rho}(X,Y;C)$ considered here are associated with complex canonical relations $C$, i.e. they are parametrised by a complex-valued phase function. Our analysis also includes the case of real canonical relations, namely, the class of Fourier integral operators with real-valued phase functions.  The nuclear trace in the sense of Grothendieck is investigated for these operators as well as the validity of the Grothendieck-Lidskii formula on Lebesgue spaces.  Criteria are presented in terms of the factorisation condition for the complex canonical relation. Necessary and sufficient conditions for the membership of Fourier integral operators in Schatten-von Neumann classes are presented in the case where the Schatten index $r>0$ belongs to the set $\mathbb{N} $ and sharp sufficient conditions are presented in the general case $r>0$. In particular, we establish necessary and sufficient conditions for the membership of Fourier integral operators to the ideal of trace class operators and to the ideal of Hilbert-Schmidt operators on $L^2(X)$. The rate of decay of eigenvalues and the trace of Fourier integral operators is also investigated in both settings, in the Hilbert space case of $L^2(X)$ using Schatten-von Neumann properties and in the context of the Banach spaces $L^p(X),$ $1<p<\infty,$ utilising the notion of $r$-nuclearity.
\end{abstract} 

\maketitle

\allowdisplaybreaks

\tableofcontents

\section{Introduction}
\subsection{Outline}
Let $(X,g_X)$ and $(Y,g_Y)$ be compact Riemannian manifolds without boundary and of dimension $n.$ Let us consider the class $I^{\mu}_{\rho, 1-\rho}(X,Y;C)$ of Fourier integral operators of type $(\rho,1-\rho)$, see H\"ormander \cite{Ho}, Duistermaat and H\"ormander \cite{Duistermaat-Hormander:FIOs-2} and Melin and Sj\"ostrand \cite{Melin:SjostrandI,Melin:SjostrandII}. 

Inspired by the problem of the distribution of eigenvalues of Fourier integral operators, in this paper we investigate the nuclearity of these operators in the sense of Grothendieck \cite{Grothendieck1955}. Also, in some cases, we present a complete characterisation (i.e. we provide necessary and sufficient conditions) on the order of these operators to guarantee their membership to the Schatten-von Neumann classes.  More precisely, we address the following problems:
\begin{itemize}
\item Given a class of Fourier integral operators $I^{\mu}_{\rho,1-\rho}(X,Y;C),$ and an operator $T$ in this class, to find `sharp' conditions on the canonical relation $C$ and order $\mu$  to guarantee the $r$-nuclearity of $T:L^{p_1}(Y)\rightarrow L^{p_2}(X) $ in the sense of Grothendieck \cite{Grothendieck1955} for any $0<r\leq 1$. Here $C$ denotes a complex canonical relation allowing also for the case of real canonical relations.
\item When $X=Y$ and $p=p_1=p_2,$ to derive information on the spectral traces of Fourier integral operators in the class $ I^{\mu}_{\rho,1-\rho}(X;C):=I^{\mu}_{\rho,1-\rho}(X,X;C),$  on the Lebesgue space
 $L^p(X),$ or equivalently, to determine the validity of the Grothendieck-Lidskii formula for the operators in the class  $ I^{\mu}_{\rho,1-\rho}(X;C).$ In other words we investigate under which conditions the nuclear trace agrees with the spectral trace, given by the sum of the eigenvalues of the operator, see Grothendieck \cite{Grothendieck1955} and Reinov and Latif \cite{Reinov1,Reinov2}. 
 \item To investigate sharp sufficient conditions on the order $\mu,$ to guarantee the inclusion of the class $I^{\mu}_{\rho,1-\rho}(X;C),$ into the Schatten-von Neumann ideal    $S_r(L^2(X))$ of order $r,$  for any $0<r<\infty,$ as well as to provide information about the distribution of eigenvalues for these operators.
\item To provide necessary and sufficient conditions on the order $\mu,$ to guarantee the membership of the class $I^{\mu}_{1,0}(X;C),$ into the ideal of trace class operators $L^2(X)$. Also, we investigate necessary and sufficient conditions on $\mu$ in order to guarantee the membership of the class $I^{\mu}_{1,0}(X;C),$ in Schatten-von Neumann classes $S_r(L^2(X))$ on $L^2(X)$ with even Schatten index $r,$ namely with $r>0$ such that $r\in 2\mathbb{Z}$. Particulary the case of $S_2(L^2(X))$ corresponds to the Hilbert-Schmidt class on $L^2(X).$    
\end{itemize}

In his classical book \cite[Chapter 10]{Barry1} Barry Simon suggested that an extension of the nuclearity results therein obtained for the $L^2$ setting of Schr\"odinger equation to $L^p$ spaces could be relevant in order to reveal other spectral properties. In this work, we follow that spirit in order to get estimates for the rate of decay of eigenvalues of Fourier integral operators on compact manifolds  acting between $L^p$ spaces, with the help of Sogge's $L^p$ bounds for the eigenfunctions  of the Laplace-Beltrami operator for compact Riemannian manifold, and the aforementioned Reinov and Latif results in \cite{Reinov1, Reinov2} for the distribution of eigenvalues of nuclear operators, which give 
an improvement of the Grothendieck inequality for the distribution of eigenvalues in the case of $L^p$ spaces.

The study of nuclearity properties and the special case of Schatten-von Neumann classes have been an object of intensive research, and made part of some of the most fundamental problems in spectral theory in the recent years, due to the new achievements from the theoretical point of view as well as in the  applications in the study of the spectrum of relevant physical problems. Indeed, 
estimates for the rate of decay of eigenvalues for double layer potentials and Neumann-Poincar\'e operators have been obtained in  
\cite{Ando-Miya1, Fuk-Miya1, Miya-Su1}. The case of Schatten-von Neumann classes for integral operators has been recently studied in \cite{DR1} and by Volkov in \cite{Volkov}. 
For a general overview about the Schatten-von Neumann properties of pseudo-differential operators and of Fourier integral operators we refer the reader to the bibliographic discussion in Subsection \ref{Bibliographic:disc}.

In the next Subsection \ref{MR}  we present our main results. For this, we are required to fix the terminology about Schatten-von Neumann classes. Let $H$ be a Hilbert space, we recall that for any $r>0,$ a compact operator $T :H\rightarrow H $ belongs to the Schatten von-Neumann ideal $S_{r}(H),$ if the sequence of its singular values $\{s_{n}(T)\}_{n\in \mathbb{N}_0}$ (formed by the eigenvalues of the operator $\sqrt{T^*T}$) belongs to $\ell^r(\mathbb{N}_0),$ that is, if  $$\Vert T\Vert_{S_r}^r:=\sum_{n=1}^\infty s_{n}(T)^r<\infty.$$ For the aspects related to the ideal of $r$-nuclear operators on Banach spaces we refer to  Subsection \ref{Nuclearity:section}.

\subsection{Main results}\label{MR} 
As it was pointed out by H\"ormander \cite{Ho}, the $L^2$-boundedness of Fourier integral operators can be reduced, via the $T^*T$ method, to the $L^2$-boundedness of pseudo-differential operators. This argument also holds if the canonical relation $C$ is complex, see Melin and Sj\"ostrand \cite{Melin:SjostrandI,Melin:SjostrandII}. A consequence of this fact is that the singular values $s_{n}(T),$ of $T\in I^\mu_{\rho,1-\rho}(X,Y;C)$ and the singular values $s_{n}(T^*T)$ of the pseudo-differential operator $T^*T,$ are related by the formula $s_n(T^*T)=s_n(T)^2.$ Since sufficient conditions for the Schatten properties of pseudo-differential operators are well understood on closed manifolds, we immediately have sufficient conditions for Schatten-von Neumann properties in the case of Fourier integral operators, see Proposition \ref{Propo:Schatten}. In consequence, our main point here is to investigate the necessary conditions for Schatten properties of these operators, see Subsection \ref{Necess:Cond:sec}. In order  to do this analysis, we are going to apply as a fundamental tool the local Weyl formula of Zelditch (see Theorem \ref{T:LWF2:PSDO})
\begin{equation}\label{LWF2:PSDO:Intro}
    \lim_{\lambda\rightarrow\infty}\frac{1}{N(\lambda)}\sum_{k: \lambda_k\leq \lambda}(P\phi_k,\phi_k)=\smallint_{T^*\mathbb{S}(X)}\sigma_P d\mu_{L},
\end{equation} where $P$ is a pseudo-differential operator of order zero, and also  the local Weyl formula due to Zelditch and Kuznecov \cite{Zelditch:Kuznecov:1992} (see Theorem \ref{T:LWF}),
\begin{equation}\label{LWF:intro}
    \lim_{\lambda\rightarrow\infty}\frac{1}{N(\lambda)}\sum_{k: \lambda_k\leq \lambda}(F\phi_k,\phi_k)=\smallint_{S(\Lambda \cap \Delta_{S^*X\times S^*X } )}\sigma_F d\mu_{L},
\end{equation}where $F\in I^{0}_{1,0}(X;C).$ In these local formulae  $(\phi_k,\lambda_k)$ are the spectral data of the operator $\sqrt{-\Delta_g},$ where $\Delta_g$ denotes the Laplacian on $(X,g).$
The local Weyl formula for pseudo-differential operators allows us to prove the following characterisation (see Theorem \ref{even:condition:PDSO}).
\begin{itemize}
    \item[(i)'] Let us consider a pseudo-differential operator  $T\in  I^\mu_{1,0}(X) $ and let $r\in \mathbb{N}= (0,\infty)\cap \mathbb{Z}.$ Then $T:L^{2}(X)\rightarrow L^{2}(X)$ belongs to the Schatten class $S_r(L^2(X)),$  if and only if  $\mu <-n/r.$ 
\end{itemize}
For Fourier integral operators, we have obtained the following criteria, see Theorems  \ref{th:trace:class} and \ref{even:condition}.
\begin{itemize}
    \item[(i)]  Let us consider a Fourier integral operator  $T\in  I^\mu_{1,0}(X;\Lambda) .$ Then $T:L^{2}(X)\rightarrow L^{2}(X)$ belongs to the Schatten class $S_1(L^2(X)),$  namely, $T$ is of trace class on $L^2,$ if and only if  $\mu <-n.$
    \item[(ii)] More generally, let us consider   $T\in  I^\mu_{1,0}(X;\Lambda) $ and let $r\in \mathbb{N}= (0,\infty)\cap \mathbb{Z}.$ Then $T:L^{2}(X)\rightarrow L^{2}(X)$ belongs to the Schatten class $S_r(L^2(X)),$  if and only if  $\mu <-n/r.$ 
\end{itemize} Observe that (i) is an special case of (ii). 
According to our approach we have proved the following implications
\begin{equation}
    \boxed{\textnormal{Local Weyl formula for pseudo-differential operators}\, \eqref{LWF2:PSDO:Intro}}\Longrightarrow \textnormal{(i)'}\Longrightarrow \textnormal{(ii).}
\end{equation}However, we also give an alternative argument  for the implications
\begin{equation}
    \boxed{\textnormal{Local Weyl formula for Fourier integral  operators}\, \eqref{LWF:intro}}\Longrightarrow \textnormal{(i)}\Longrightarrow \textnormal{(ii)}.
\end{equation}
First,  for a Fourier integral operator $T\in S_r(L^2(X)),$ we have the following estimate about the growth of its eigenvalues $\lambda_j(T),$ see Theorem \ref{th:trace:class:dist},
\begin{equation}\label{j:r:intro}
    |\lambda_j(T)|=O(j^{-\frac{1}{r}}).
\end{equation}  On the other hand, let us discuss the sharpness of (i) and (ii).
Let $r>0$ and let  $T\in  I^\mu_{1,0}(X;\Lambda) $ be elliptic, then $T^*T$ is an elliptic pseudo-differential operator  of order $2\mu.$ In this case, it is very well known that $T:L^{2}(X)\rightarrow L^{2}(X)$ belongs to the Schatten class $S_r(L^2(X)),$  if and only if $T^*T$ belongs to the Schatten class $S_{\frac{r}{2}}(L^2(X)),$ which holds if and only if $\mu <-n/r,$ showing that the conditions in (i) and (ii) are sharp when $T$ is elliptic, and leaving open the following question:  
\begin{openproblem}\label{open:problem}
    Prove (or disprove) that if  $T\in I^\mu_{1,0}(X;\Lambda) ,$  is a non-elliptic Fourier integral operator that belongs to the Schatten class  $ S_{r}(L^2(X)),$ then $\mu<-n/r,$ when $r\in (0,1)\cup ((1,\infty)\setminus \mathbb{Z}).$
\end{openproblem} To the best of our knowledge, under the conditions in Open problem \ref{open:problem}, one has the validity of the order inequality $\mu\leq -n/r,$ see Lemma \ref{Invariance:trace:2}. 

On the other hand, the problem of the distribution of eigenvalues of operators on Banach spaces, particularly in the case of Fourier integral operators on the Lebesgue spaces $L^p(X)$ can be addressed using the notion of $r$-nuclearity due to Grothendieck \cite{Grothendieck1955}, see Subsection \ref{Nuclearity:section} for this terminology. Indeed, Grothendieck's notion of nuclearity allows the analysis of traces of operators on Banach spaces. In our analysis, we are going to exploit the ideal properties of the class of $r$-nuclear operators. Let $\Delta_g$ be the Laplace operator on $(X,g).$ According to our methodology, we first classify the membership of the powers of the operators $$E=\sqrt{1-\Delta_g}^{-s},\quad  s\in \mathbb{R},$$ in the  ideal of $r$-nuclear operators from $L^{p_1}(X)$ into $L^{p_2}(X).$ The kernel of the operator $E=\sqrt{1-\Delta_g}^{-s}$ has the form, see Subsection \ref{sect:bessel},
$$ K(x,y)=\sum_{\ell=0}^{\infty}\sum_{k=1}^{d_\ell} \lambda_\ell^{-s}e_{\ell}^{k}(x)\overline{e_{\ell}^{k}(y)},$$ with the systems $e_{\ell}^{k}$ denoting its eigenfunctions, and the $\lambda_\ell$'s are the corresponding eigenvalues of  $\sqrt{1-\Delta_g}.$ In consequence, the analysis of the $r$-nuclearity of $E$ involves the size of the $L^p$-norms of the eigenfunctions  $e_{\ell}^{k}.$ Imminently, a sharp analysis in this direction should apply the global $L^p$-estimates for the eigenfunctions of  $\sqrt{1-\Delta_g}$ due to Sogge  \cite{Sogge1988}, see Subsection \ref{Sogge:section}. These estimates have the form
\begin{equation}
    \Vert \phi\Vert_{L^p(X)}\lesssim \lambda^{\varphi(p)},\,1\leq p\leq \infty,
\end{equation}where $(\phi,\lambda)$ is a spectral datum of the operator $\sqrt{1-\Delta_g}$ with the eigenfuction $\phi$ being $L^2$-normalised. Combining these properties as well as the $L^p$-$L^q$-boundedness of Fourier integral operators we derive our criteria for the $r$-nuclearity of Fourier integral operators, see Subsection \ref{Main:nuclearity:section}. In this setting, we note that the structure of the order $\mu$ guarantying the $r$-nuclearity of an operator is the following: firstly, $\mu$ satisfies an inequality of the type
\begin{equation}\label{structure:mu}
     \mu < \mu(\varkappa)-\frac{n}{r}-\varphi(p_2)-\varphi(p_1'),\quad \varkappa\geq 0.
\end{equation}
The parameter $\mu(\varkappa)$ comes from the $L^{p_2}$-$L^{p_2+\varkappa}$ boundedness of Fourier integral operators. We revisit this topic in Subsection \ref{lp:lq}, where we establish these boundedness properties for the classes $I^\mu_{\rho,1-\rho}(X,Y;C)$ by combining the methods of \cite[Pages 39-40]{Ruzhansky:CWI-book} and the approach of  H\"ormander \cite{Hor67}, in terms of the real factorisation condition introduced by Seeger, Sogge and Stein in \cite{SSS}. In the context of complex canonical relations, we have presented this condition as developed by the third author in \cite{Ruzhansky:CWI-book}. On the other hand, the quantity $-\frac{n}{r}-\varphi(p_2)-\varphi(p_1')$ in \eqref{structure:mu} comes from the $r$-nuclearity of the operator $E=\sqrt{1-\Delta_g}^{-s}:L^{p_1}(X)\rightarrow L^{p_2}(X) ,$ where the inequality $s>\frac{n}{r}+\varphi(p_2)+\varphi(p_1')$ arises, see Proposition  \ref{Propo:1}. With the order $\mu(\varkappa),$ allowing the $L^{p_2}\rightarrow L^{p_2+\varkappa} $ boundedness of $TE^{-1} $ and with $E:L^{p_1}\rightarrow L^{p_2} $ being $r$-nuclear, the $r$-nuclearity of $T :L^{p_1}\rightarrow L^{p_2+\varkappa}$ follows from the ideal properties of the $r$-nuclear operators. We note that in the $L^2$-case, namely, the condition on $\mu$ in \eqref{structure:mu} allowing the $r$-nuclearity of $T:L^2\rightarrow L^2,$ is saturated to the inequality $\mu<-n/r.$ This inequality is compatible with the order condition on $\mu$ allowing the membership of the class $I^\mu_{\rho,1-\rho}(X,C)$ to the Schatten classes, since the set of $r$-nuclear operators on $L^2$ agrees with the Schatten class due to a result by Oloff \cite{Oloff1970}.

Finally, as an application of the nuclearity results above, when investigating the distribution of eigenvalues of a $r$-nuclear Fourier integral operator $T,$ we have obtained the rate of decay
\begin{equation}\label{dist:lp}
    |\lambda_j(T)|=O(j^{ -\frac{1}{r}+\left|\frac{1}{2}-\frac{1}{p}\right|}),
\end{equation} where $\lambda_j(T),j\in \mathbb{N}_0$ denotes the system of eigenvalues of $T: L^{p}(X)\rightarrow L^{p}(X),$ see Theorem \ref{dist:2}, while the Grothendieck-Lidskii formula has been obtained in Theorem \ref{dist:1}. Note that for $0<r\leq 1,$ we have a gain of information in \eqref{dist:lp} in terms of the Lebesgue index $p=p_2+\varkappa,$ when comparing it with the information provided in the setting of the Schatten-von Neumann classes, see \eqref{j:r:intro}.

\subsection{Bibliographic discussion}\label{Bibliographic:disc} The problem of providing necessary and sufficient conditions for
Schatten-von Neumann properties of Fourier integral operators has been considered mainly on $\mathbb{R}^n,$ see e.g. Bishop \cite{BishopACHA2011}, Concetti and Toft \cite{ConcettiToft2009} and Concetti, Garello, and Toft \cite{ConcettiGarelloToft2010}. 
However, in a special class of Fourier integral operators, namely, the class of pseudo-differential operators has been intensively investigated. Indeed, to find sharp sufficient conditions on a pseudo-differential operator to allow its membership to the Schatten classes on compact manifolds  has been an active area of research, e.g.  we refer the reader to the works of the second and the third author \cite{DelRuz2014JFA,DelRuz2014CRAS,DelRuzJMPA2017,DelRuzMRL2017,DelgadoRuzhanskyJAM2018} and the references therein.

The $r$-nuclearity of integral operators on Lebesgues spaces, and on compact manifolds have been consistently investigated mainly by the second and third authors. This program has started with the work \cite{DelRuz2014JMPA}, we refer the reader to further developments to the references \cite{Delgado2015Tohoku,DelRuz2016JST,DelRuz2016JLMS,DelgadoRuzhanskyJAM2018}.   Recently, these results have been revisited in terms of symbol criteria and also order criteria in  \cite{CardonaRuzToft} for pseudo-differential operators on compact Lie groups where there were presented several open problems and conjectures with the local Weyl formula on compact Lie groups in \cite{CarDelRuzJLT2024} as the main ingredient for the necessary conditions provided in this setting. For the conditions of $r$-nuclearity on Besov spaces on compact Lie groups we refer to \cite{Cardona2017JFAA} and for the nuclearity of multilinear operators on the torus, we refer the reader to \cite{CardonaKumarJFAA2019}. For applications of the theory of vector-valued pseudo-differential operators to index theory, we refer to \cite{CardonaKumarMANA2021}.

\subsection{Structure of the manuscript}
The organisation of this paper is as follows. 
In Section \ref{preliminaries} we present the preliminaries used in this article, namely, we discuss the $L^p$-boundedness of Fourier integral operators and then the real factorisation condition of Seeger, Sogge and Stein \cite{SSS}, and its extension to complex canonical relations \cite{Ruzhansky:CWI-book}. Also, we record the $L^p$-bounds for eigenfunctions of the Laplacian due to Sogge \cite{Sogge1988} and we present the notions of $r$-nuclearity of operators on Banach spaces as developed by Grothendieck \cite{Grothendieck1955}.  In Section \ref{Nuc:sec} we investigate the $r$-nuclearity of Fourier integral operators and the distribution of eigenvalues of these operators on $L^p$-spaces. 
Finally, in Section \ref{Schatten:sec} the problem of determining necessary and sufficient conditions for Fourier integral operators to belong to the Schatten-von Neumann ideals is addressed.







\section{Preliminaries}\label{preliminaries} In this section we introduce the basics of Fourier integral operators with complex-valued phase functions, the fundamental global estimates by Sogge \cite{Sogge1988} and the notion of $r$-nuclearity in the sense of Grothendieck \cite{Grothendieck1955} to be used in our further analysis.  As for the analysis of Fourier integral operators, we follow the notions and terminology introduced by H\"ormander \cite{Ho}, Duistermaat and H\"ormander \cite{Duistermaat-Hormander:FIOs-2} and Melin and Sj\"ostrand \cite{Melin:SjostrandI,Melin:SjostrandII}. Here, for all $\mu\in \mathbb{R},$ on a compact manifold $M$ without boundary,  $\Psi^\mu(M):=\Psi^\mu_{1,0}(M)$ denotes the Kohn-Nirenberg class of pseudo-differential operators of order $\mu,$ see H\"ormander \cite{Hormander1985III}.

\subsection{Lagrangian manifolds} Let us follow the notation in  \cite[Chapter I]{Ruzhansky:CWI-book}. For a compact manifold $M,$ we denote by $TM$ and $T^*M$ the tangent bundle, and cotangent bundle, respectively.
A $2$-form $\omega$ is called {\it symplectic}  on $M$ if $$d\omega =0,$$ and for all $x\in M,$ the bilinear form $\omega_x$  is antisymmetric and non-degenerate on $T_{x}M$. The canonical symplectic form $\sigma_M$ on $M$ can be defined as follows. Let $\pi:=\pi_M:T^*M\rightarrow M,$  be the canonical projection. For any $(x,\xi)\in T^*M,$ let us consider the linear mappings
\begin{equation*}
    d\pi_{(x,\xi)}:T_{(x,\xi)}(T^*M)\rightarrow T_{x}M\textnormal{   and,   }\xi:T_{x}M\rightarrow \mathbb{R}.
\end{equation*}The composition $\alpha_{(x,\xi)}:=\xi\circ d\pi_{(x,\xi)}\in T^{*}_{(x,\xi)}(T^*M),$ that is
$
   \xi\circ d\pi_{(x,\xi)}:T_{(x,\xi)}(T^*M)\rightarrow  \mathbb{R},
$ defines  a 1-form $\alpha $ on $T^*M.$ With the notation above, the canonical symplectic form $\sigma_M$ on $M$ can be defined via 
\begin{equation}\label{sigmaM}
    \sigma_M:=d\alpha.
\end{equation}It follows that $d\sigma_M=0$ and then  that $\sigma_M$ is symplectic. If $M=X\times Y,$ we have that  $\sigma_{X\times Y}=\sigma_X\oplus -\sigma_Y .$ Now we record the family of submanifolds that are necessary when one defines the canonical relations.
\begin{definition} Let $M$ and $X$ be of dimension $n.$ We define a  Lagrangian submanifold of $M$, a conic manifold of $T^*X,$ and the conormal bundle of a submanifold of $X$, as follow:
\begin{itemize}
    \item  A submanifold $\Lambda\subset T^*M$ of dimension $n$ is called {\it Lagrangian} if 
\begin{equation*}
    T_{(x,\xi)}\Lambda=( T_{(x,\xi)}\Lambda)^{\sigma}:=\{v\in T_{(x,\xi)}(T^*M):\sigma_M(v,v')=0,\,\forall v'\in T_{(x,\xi)}\Lambda \}.
\end{equation*}
\item We say that  $\Lambda\subset T^*M\setminus 0$ is {\it conic} if $(x,\xi)\in \Lambda,$ implies that $(x,t\xi)\in \Lambda,$ for all $t>0.$ 
\item Let $\Sigma\subset X$ be a smooth submanifold of $X$ of dimension $k.$ Its conormal bundle in $T^*X$ is defined by
\begin{equation}
    N^{*}\Sigma:=\{(x,\xi)\in T^*X: \,x\in \Sigma,\,\,\xi(\delta)=0,\,\forall \delta\in T_{x}\Sigma  \}.
\end{equation}
\end{itemize}    
\end{definition}
The following facts characterise the conic Lagrangian submanifolds of $T^*M.$
\begin{itemize}
    \item Let $\Lambda\subset T^*M\setminus 0,$ be a closed sub-manifold of dimension $n.$ Then $\Lambda$ is a conic Lagrangian manifold if and only if the 1-form $\alpha$ in \eqref{sigmaM} vanishes on $\Lambda.$
    \item Let $\Sigma\subset X,$ be a submanifold of dimension $k.$ Then its conormal bundle $N^*\Sigma$ is a conic Lagrangian manifold. 
\end{itemize}
\begin{remark}
   The Lagrangian manifolds have the following property.
 Let $\Lambda\subset T^{*}M\setminus 0,$ be a conic Lagrangian manifold and let 
    \begin{equation}
        d\pi_{(x,\xi)}:T_{(x,\xi)}\Lambda\rightarrow T_{x}M,
    \end{equation}have constant rank equal to $k,$ for all $(x,\xi)\in \Lambda.$ Then, each $(x,\xi)\in \Lambda$ has a conic neighborhood $\Gamma$ such that 
    \begin{itemize}
        \item[1.] $\Sigma=\pi(\Gamma\cap \Lambda)$ is a smooth manifold of dimension $k.$
        \item[2.] $\Gamma\cap \Lambda$ is an open subset of $N^*\Sigma.$
    \end{itemize}
\end{remark}

The Lagrangian manifolds have a local representation defined in terms of {\it phase functions} that can be defined as follows. For this, let us consider a local trivialisation $M\times (\mathbb{R}^n\setminus 0), $ where we can assume that $M$ is an open subset of $\mathbb{R}^n.$

\begin{definition}[Real-valued phase functions]\label{Real:defin} Let $\Gamma$ be a cone in $M\times (\mathbb{R}^N\setminus 0). $ A smooth function $\phi:M\times (\mathbb{R}^N\setminus 0)\rightarrow \mathbb{R}, $ $(x,\theta)\mapsto \phi(x,\theta),$ is a real {\it phase function} if it is homogeneous of degree one in $\theta$ and has no critical points as a function of $(x,\theta),$ that is 
\begin{equation}
 \forall t>0,\,   \phi(x,t\theta)=t\phi(x,\theta), \textnormal{   and   }d_{(x,\theta)}\phi(x,\theta)\neq 0,\forall (x,\theta)\in M\times (\mathbb{R}^N\setminus 0).
\end{equation}Additionally, we say that $\phi$ is a  {\it non-degenerate phase function} in $\Gamma,$ if for any $(x,\theta)\in \Gamma$ such that $d_{\theta}\phi(x,\theta)=0,$ one has that 
\begin{equation}
    d_{(x,\theta)}\frac{\partial \phi}{\partial\theta_{j}}(x,\theta),\,\,1\leq j\leq N,
\end{equation}is a system of linearly independent vectors on $\mathbb{R}$.
\end{definition}
The following facts describe locally a Lagrangian manifold in terms of a phase function.
\begin{itemize}
    \item Let $\Gamma$ be a cone in $M\times (\mathbb{R}^N\setminus 0), $ and let $\phi$ be a non-degenerate phase function in $\Gamma.$ Then, there exists an open cone $\tilde{\Gamma}$ containing $\Gamma$ such that the set
    \begin{equation}
        U_{\phi}=\{(x,\theta)\in \tilde{\Gamma}:d_\theta \phi(x,\theta)=0\},
    \end{equation}is a smooth conic  sub-manifold of $M\times (\mathbb{R}^N\setminus 0)$ of dimension $n.$ The mapping 
    \begin{equation}
        L_{\phi}: U_\phi\rightarrow T^*M\setminus 0,\,\, L_{\phi}(x,\theta)=(x,d_x\phi(x,\theta)),
    \end{equation}is an immersion. Let us denote $\Lambda_\phi=L_{\phi}(U_\phi).$
    \item Let $\Lambda\subset T^*M\setminus 0$ be a sub-manifold of dimension $n.$ Then,  $\Lambda$ is a conical Lagrangian manifold if and only if  every $(x,\xi)\in \Lambda$ has a conic neighborhood $\Gamma$ such that $\Gamma\cap \Lambda=\Lambda_\phi, $ for some non-degenerate phase function $\phi.$ 
\end{itemize}
\begin{remark}
The cone condition on $\Lambda$ corresponds to the homogeneity of the phase function.
\end{remark}
\begin{remark} Although we have presented the previous Definition \ref{Real:defin} of non-degenerate real phase function in the case of a real function of  $(x,\theta),$ the same can be defined if one considers functions of $(x,y,\theta).$ Indeed, a real-valued phase function $\phi(x,y,\theta)$ homogeneous of order 1 at $\theta\neq 0$  that satisfies the following two conditions
 \begin{equation}
        \textnormal{det}\partial_x \partial_\theta(\phi(x,y,\theta))\neq 0,\,\, \textnormal{det}\partial_y \partial_\theta(\phi(x,y,\theta))\neq 0,\,\,\theta\neq 0,
    \end{equation} is called  non-degenerate. 
\end{remark}
\begin{remark}
 For a symplectic manifold $M$ of dimension $2n,$ we will denote by $\widetilde{M}$ its almost analytic continuation in $\mathbb{C}^{2n}$ (see   \cite[Page 10]{Ruzhansky:CWI-book} for details about the construction of $\widetilde{M}.$ For completeness, we present such a notion in the following subsection.
\end{remark}
\subsection{Almost analytic continuation of real manifolds}\label{2:2} To define the almost analytic continuation of a manifold we require some preliminaries. We record that a function $f:U\to \mathbb{C}$ defined on an open subset $U\subset \mathbb{C}^n$ is called {\it almost analytic} in $U_{\mathbb{R}}:=U\cap \mathbb{R}^n,$ if $f$ satisfies the Cauchy-Riemann equations on $U_{\mathbb{R}},$ that is, if $\overline{\partial}f$ and all its derivatives vanish in $U_\mathbb{R}.$ Here, as usual, $$\overline{\partial}=1/2(\partial_x+i\partial_y).$$

One defines an {\it almost analytic extension} of a manifold $M$ requiring that the corresponding coordinate functions are almost analytic in $M.$ Here, we are going to present this notion in detail. 

Let $\Omega\subset \mathbb{R}^n$ and let $\rho:\Omega\rightarrow \mathbb{R}$ be a non-negative and Lipschitz function. A function $f:\Omega\rightarrow \mathbb{R}$ is called $\rho$-flat in $\Omega$ if for every compact set $K\subset \Omega,$ and for every integer $N\geq 0,$ one has the estimate $$|f(x)|\lesssim_{K,N} \rho(x)^{N},$$ for all $x\in K.$ This notion defines an equivalence relation on the space of mapping on $\Omega:$
$$f\textnormal{ and }g \textnormal{ are }\rho\textnormal{-equivalent if }f-g\textnormal{ is }\rho\textnormal{-flat on }\Omega. $$ If $K_0\subset \Omega$ is a compact set, we will say that $f$ is {\it flat } on $K_0,$ if it is $\rho$-flat with $\rho(x):=\textnormal{dist}(x,K_0).$ Then, one has the following property:
\begin{itemize}
    \item let $f\in C^{\infty}(\Omega)$ be $\rho$-flat. Then all its derivatives are $\rho$-flat and $f$ is flat on $K_0$ if and only if $D^{\alpha}f=0$ for all $x\in K_0$ and all $\alpha.$ 
\end{itemize}
Let $G$ be an open subset of $\mathbb{C}^n$ and let $K$ be a closed subset of $G.$ A  function $f\in C^\infty(G)$ is {\it almost analytic} on $K$ if the functions $\partial_j f$ are flat on $K,$ for all $1\leq j\leq n.$ For an open set  $\Omega\subset \mathbb{R}^n,$ we will denote
\begin{equation}
   \tilde{\Omega}:=\Omega+i\mathbb{R}^n.
\end{equation}One can make the identification $$\Omega\cong \tilde{\Omega}\cap\{z:\textnormal{Im}(z)=0 \}.$$ Note that
every function $f\in C^{\infty}(G_{\mathbb{R}})$ defines an equivalence class of almost analytic functions on $G_\mathbb{R},$ which consist of functions in $C^\infty(G)$ which are almost analytic in $G_\mathbb{R},$ modulo functions being flat on $G_{\mathbb{R}}.$ Any representative of this class is called an {\it almost analytic continuation of }$f$ in $G.$

Let $O$ be an open subset of $\mathbb{C}^n.$ Let $M$ be a smooth submanifold of codimension $2k$ of $O,$ and let $K$ be a closed subset of $O.$ Then, $M$ is called {\it almost analytic }on $K,$ if for any point $z_0\in K,$ there exists an open set $U\subset O,$ and $k$-functions $f_{j},$ such that every $f_j$ is almost analytic on $K\cap U,$
and such that in $U,$ $M$ is defined by the zero-level sets $f_j(z)=0,$ with the differentials $df_{j},$ $1\leq j\leq k,$ being linearly independent over $\mathbb{C}.$

Two almost analytic submanifolds $M_1$ and $M_2$ of $O$ are {\it equivalent} if they have the same dimension, if $$M_{1,\mathbb{R}}=M_1\cap \mathbb{R}^n=M_2\cap \mathbb{R}^n=M_{2,\mathbb{R}}=:M_{\mathbb{R}},$$ and locally $f_j-g_j$ are flat functions on $M_{\mathbb{R}}$ where $f_j$ and $g_j$ are the functions that define $M_1$ and $M_2,$ respectively. 

A real manifold $\Omega$ defines an equivalence class of almost analytic manifolds in $\tilde{\Omega}.$ A representative of this class is called an {\it almost analytic continuation } of $\Omega$ in $\mathbb{C}^n.$
Now, let us consider:
\begin{itemize}
    \item a real symplectic manifold $M$ of dimension $d=2n,$ and let $\tilde{M}$ be (modulo its equivalence class) its almost analytic continuation in $\mathbb{C}^{2n}.$
    \item Let $\Lambda\subset \tilde{M}$ be an almost analytic sub-manifold containing the real point $\rho_0\in M,$ and let $(x,\xi)$ be its real symplectic coordinates in a neighbourhood $W$ of $\rho_0.$ 
    \item Let $(\tilde{x},\tilde{\xi})$ be almost analytic continuations of the coordinates $(x,\xi)$ in $W,$ in such a way that $(\tilde{x},\tilde{\xi})$ maps $ \tilde{W}$  diffeomorphically on an open subset of $\mathbb{C}^{2n}.$
    \item Let $g$ be an almost analytic function such that $\textnormal{Im}(g)\geq 0,$ in $\mathbb{R}^n,$ and such that $\Lambda$ is defined in a neighbourhood of $\rho_0$ by the equations $\tilde{\xi}=\partial_{\tilde{x}}g(x),$ $\tilde{x}\in \mathbb{C}^n.$
\end{itemize}
Then, an almost analytic manifold $\Lambda$ satisfying these properties, in some real symplectic coordinate system at every point is called a {\it positive Lagrangian manifold}. We conclude this discussion with the following result,  see e.g. Theorem 1.2.1 in \cite{Ruzhansky:CWI-book}.
\begin{proposition}
    Let $M,$ $\Lambda$ and $W$ be as before. If $(\tilde y, \tilde \eta)$ is another almost analytic continuation of coordinates in $\tilde{W}$ and $\Lambda$ is defined by the equation $\tilde\eta=H(\tilde y),$ in a neighbourhood of $\rho_0,$ then $\Lambda$ is locally equivalent to the manifold $\tilde{\eta}=\partial_{\tilde y}h,$ $\tilde{y}\in \mathbb{C}^n,$ where $h$ is an almost analytic function and $\textnormal{Im}(h)\geq 0.$
\end{proposition}

\subsection{Fourier integral operators with real-valued phases}

We can assume that  $X, Y$ are open sets in $\mathbb{R}^n$. One defines the class of
  Fourier integral operators $T\in I^\mu_{\rho}(X, Y;\Lambda)$ by
  the (microlocal) formula
   \begin{equation}\label{EQ:FIO}
     Tf(x)=\smallint\limits_Y\smallint\limits_{\mathbb{R}^N}
   e^{i\Psi(x,y,\theta)} a(x,y,\theta) f(y)d\theta\;dy,
   \end{equation} 
   where the symbol $a$ is a smooth function locally  in the class $S^\mu_{\rho,1-\rho}(X\times Y\times  (\mathbb{R}^n\setminus 0) ),$ with $1/2\leq \rho\leq 1.$ This means that  $a$ satisfies the symbol inequalities
   $$ |\partial_{x,y}^{\beta}\partial_\theta^\alpha
   a(x,y,\theta)|
    \leq C_{\alpha\beta}(1+|\theta|)^{\mu-\rho|\alpha|+(1-\rho)|\beta|},$$
   for $(x,y)$ in any compact subset $K$ of $X\times Y,$ and $\theta \in \mathbb{R}^N\setminus 0,$ while the real-valued phase function $\Psi$ satisfies 
     the following properties:
\begin{itemize}
    \item[1.] $\Psi(x,y,\lambda\theta)=
      \lambda\Psi(x,y,\theta),$ for all $\lambda>0$;
    \item[2.] $d\Psi\not=0$;
    \item[3.] $\{d_\theta\Psi=0\}$ is smooth 
      (e.g. $d_\theta\Psi=0$ implies
      $d_{(x,y,\theta)}\frac{\partial\Psi}{\partial\theta_j}$ are linearly
      independent).
\end{itemize}
Here $\Lambda \subset T^*(X\times Y)\setminus 0$ is a Lagrangian manifold locally parametrised by the phase function $\Psi,$
 $$ \Lambda=\Lambda_\Psi=\{(x,d_x\Psi,y,d_y\Psi): d_\theta\Psi=0\}.$$
\begin{remark}
  The   canonical relation associated with $T$
  is the conic Lagrangian manifold in 
  $T^*(X\times Y)\backslash 0$,  defined by
  $\Lambda'=\{(x,\xi,y,-\eta): (x,\xi,y,\eta)\in \Lambda\}.$
In view of the H\"ormander equivalence-of-phase-functions theorem (see e.g. Theorem 1.1.3 in \cite[Page 9]{Ruzhansky:CWI-book}), the notion of Fourier integral operator becomes independent of the choice of a particular phase function associated to a Lagrangian manifold $\Lambda.$  Because of the diffeomorphism $\Lambda\cong \Lambda',$ we do not distinguish between $\Lambda $ and $\Lambda'$ by saying also that $\Lambda$ is the canonical relation associated with $T.$  
\end{remark}

\subsection{Fourier integral operators with complex-valued phases}\label{2:4}
Now we record the following result for canonical relations, see e.g. Theorem 1.2.1 in \cite{Ruzhansky:CWI-book}.
\begin{proposition}
    Let $\Psi$ be a phase function of the positive type that is defined in a conical neighbourhood. Let $\tilde{\Psi}$ be an almost analytic homogeneous continuation of $\Psi$ in a canonical neighbourhood in $\mathbb{C}^n\times \mathbb{C}^n\times ({\mathbb{C}^{n}}\setminus \{0\}).$ Let \begin{equation}
        C_{\tilde \Psi}=\{(\tilde x, \tilde y, \tilde \theta)\in \mathbb{C}^n\times \mathbb{C}^n\times ({\mathbb{C}^{n}}\setminus \{0\}):\partial_{\theta}\tilde\Psi(\tilde x, \tilde y, \tilde \theta)=0 \}.
    \end{equation}Then, the image of the set $C_{\tilde \Psi}$ under the mapping 
    \begin{equation}
      C_{\tilde \Psi}  \ni(\tilde x, \tilde y, \tilde \theta)\mapsto (\tilde x, \partial_{\tilde x}\tilde\Psi,\partial_{\tilde y}\tilde\Psi )|_{(\tilde x, \tilde y, \tilde \theta)}\in \mathbb{C}^n\times ({\mathbb{C}^{n}}\setminus \{0\})\times ({\mathbb{C}^{n}}\setminus \{0\})
    \end{equation} is a local positive Lagrangian manifold.
\end{proposition}
Let $X$ and $Y$ be smooth manifolds of dimension $n$ and let us denote by $\widetilde{ T^*(X\times Y)\setminus 0}$  the  almost analytic continuation  of  $ T^*(X\times Y)\setminus 0.$ Let   $$  C\equiv C_{\Phi}=\{(x,d_x\Phi,y, d_y\Phi):d_\theta \Phi=0\}\subset\widetilde{ T^*(X\times Y)\setminus 0}$$  be a {\it smooth positive homogeneous canonical relation.} This means that $C$ is locally parametrised by a complex phase function $\Phi(x,y,\theta),$ satisfying the following properties:
\begin{itemize}
    \item[1.] $\Phi$ is homogeneous of order one in $\theta:$  $\Phi(x,y,t\theta)=t\Phi(x,y,\theta),$ for all $t>0,$
    \item[2.] $\Phi$ has no critical points on its domain, i.e. $d\Phi\neq 0;$
    \item[3.] $\{d_\theta \Phi(x,y,\theta)=0\}$ is smooth (e.g. $d_\theta \Phi=0$ implies that $d\frac{\partial \Phi}{\partial\theta_j}$ are independent over $\mathbb{C}$),
    \item[4.] $\textnormal{Im}(\Phi)\geq 0.$
\end{itemize}When the last condition is satisfied one says that the complex phase function $\Phi$ is of {\it positive type}. Then, the class  $I^\mu_\rho(X,Y;C)$ of Fourier integral operators $T$ is determined modulo $C^{\infty}$ by those integral operators $T$ that locally have an integral  kernel of the form
\begin{align*}
    A(x,y)=\smallint\limits_{\mathbb{R}^N}e^{i\Phi(x,y,\theta)}a(x,y,\theta)d\theta,
\end{align*}where $a(x,y,\theta)$ is a symbol of order $\mu+\frac{N-n}{2}.$ 
By the equivalence-of-phase-function theorem, we can always assume that $N=n,$ and that the symbol $a$ satisfies the type $(\rho,1-\rho)$-estimates
\begin{equation}
    |\partial_{x,y}^{\beta}\partial_\theta^\alpha a(x,y,\theta)|\leq C_{\alpha,\beta}(1+|\theta|)^{\mu-\rho|\alpha|+(1-\rho)|\beta|},\,\,\frac{1}{2}\leq \rho\leq 1,
\end{equation}locally uniformly in $(x,y).$

\subsection{$L^p$-boundedness of Fourier integral operators}
First, let us review the mapping properties of Fourier integral operators in the case of real phase functions.  By following \cite{Duistermaat-Hormander:FIOs-2,Ho}, these classes are denoted by $I^\mu_{\rho,1-\rho}(X,Y;\Lambda),$ with $\Lambda$ being  considered locally a graph of a symplectomorphism from $T^{*}X\setminus 0$ to  $T^{*}Y\setminus 0,$ which are equipped with the canonical symplectic forms $d\sigma_X$ and $d\sigma_Y,$ respectively. Such Fourier integral operators are called non-degenerate.  We recall that the symplectic structure of $\Lambda$ is determined by the symplectic $2$-form $\omega$ on $X\times Y,$ $\omega=\sigma_X\oplus -\sigma_Y .$ Let $\pi_{X\times Y}$ be the canonical projection from $T^*X\times T^*Y$ into $X\times Y.$ As in the case of pseudo-differential operators, non-degenerate Fourier integral operators of order zero are bounded on $L^2.$ The fundamental work of Seeger, Sogge and Stein \cite{SSS} establishes the boundedness of 
 $T\in I^\mu_1(X,Y;\Lambda),$ from $L^p_{\textnormal{comp}}(Y)$ into $L^p_{\textnormal{loc}}(X)$ with the order
 \begin{equation}\label{critical:SSS}
     \mu\leq -(n-1)|1/2-1/p|,
 \end{equation}if $1<p<\infty,$ and from $H^1_{\textnormal{comp}}(Y)$ into $L^1_{\textnormal{loc}}(X)$ if $p=1.$ Also, for $p=1$ in \eqref{critical:SSS} and as a consequence of the weak (1,1) estimate in  Tao \cite{Tao:weak11}, an operator $T$ of order $-(n-1)/2$ is locally of weak (1,1) type. This weak (1,1) estimate for Fourier integral operators with complex phases has been proved by the first and the last author in \cite{CardonaRuzhanskyFIO2024}.
 
The critical Seeger-Sogge-Stein order \eqref{critical:SSS} is sharp if $d\pi_{X\times Y}|_{\Lambda}$ has full rank equal to $2n-1$ somewhere and if $T$ is an elliptic operator. When  $d\pi_{X\times Y}|_{\Lambda}$ does not attain the maximal rank $2n-1,$ the upper bound for the order \eqref{critical:SSS} is not sharp and may depend of the singularities of $d\pi_{X\times Y}|_{\Lambda}.$  In conclusion, as it was observed in  \cite{SSS}, the mapping properties of the classes $ I^\mu_{\rho,1-\rho}(X,Y;\Lambda)$ of Fourier integral operators depend of the singularities and of the maximal rank of the canonical projection. So, an additional condition on the canonical relation $\Lambda$ was introduced in  \cite{SSS}, the so called, factorisation condition for $\pi_{X\times Y}$. Roughly speaking, it can be introduced as follows. 
\begin{definition}[The real factorisation condition]
    Assume that there exists $k\in \mathbb{N},$ with $0\leq k\leq n-1,$ such that for any $\lambda_0=(x_0,\xi_0,y_0,\eta_0)\in \Lambda,$ there is a conic neighborhood $U_{\lambda_0}\subset \Lambda$ of $\lambda_0,$ and a smooth homogeneous of order zero map $\pi_{\lambda_0}:U_{\lambda_0}\rightarrow \Lambda,$ such that 
 \begin{equation}\label{FC:Intro:1}
    \textnormal{(RFC):   } \textnormal{rank}(d\pi_{\lambda_0})=n+k,\textnormal{   and   }\pi_{X\times Y}|_{U_{\lambda_0}}=\pi_{X\times Y}|_{\Lambda}\circ \pi_{\lambda_0}.
\end{equation}Then we say that $\Lambda$ satisfies the {\it real factorisation condition} of rank $k.$ 
\end{definition}

Under the real factorisation condition (RFC) in  \eqref{FC:Intro:1},  Seeger, Sogge and Stein in   \cite{SSS} proved that the order
 \begin{equation}
     \mu\leq -(k+(n-k)(1-\rho))|1/2-1/p|,
\end{equation}guarantees that any Fourier integral operator $T\in I^\mu_{\rho,1-\rho}(X,Y;\Lambda),$ with $\rho\in [1/2,1],$ is bounded from $L^p_{\textnormal{comp}}(Y)$ into $L^p_{\textnormal{loc}}(X),$ and for $p=1,$ from $H^1_{\textnormal{comp}}(Y)$ into $L^1_{\textnormal{loc}}(X).$  All this concerns Fourier integral operators  with real phase functions.

The theory of Fourier integral operators with complex phases $T\in  I^\mu_{\rho,1-\rho}(X,Y;C) $ was systematically developed by Melin and Sj\"ostrand \cite{Melin:SjostrandI,Melin:SjostrandII}. The $L^2$-boundedness for these Fourier integral operators and with order zero was proved in \cite{Melin:SjostrandI}, and with greater generality by H\"ormander \cite{Ho2}. There is a rich structure when working with complex phases instead of the real ones.  First, there are no geometric obstructions like the non-triviality of the Maslov cohomology class. Second, one can use a single complex phase function to give a global parametrisation of a Fourier integral operator (see Laptev, Safarov and Vassiliev \cite{LaptevSafarov}).

We consider the complex canonical relation $C$ satisfying the {\it complex factorisation condition} introduced by the third author in \cite{Ruzhansky:CWI-book}, which we introduce as follows. Let us employ the tilde in $\widetilde{M}$ to  denote an almost analytic extension of a real manifold $M$, (we refer to \cite{Melin:SjostrandII}  for its extensive discussion and analysis).
\begin{definition}[The complex factorisation condition] Let  $C$ be  locally parametrised  by a complex phase function $\Phi,$ such that 
\begin{equation}
       \Phi_\tau:= \textnormal{Re}(\Phi)+\tau\textnormal{Im}(\Phi)
\end{equation} defines a local graph in $T^{*}({X\times Y})\setminus 0,$  for some $\tau\in \mathbb{R},$ for instance, with $|\tau|<\sqrt{3},$ for technical reasons, and such that the real-valued phase function $\Phi_\tau,$ satisfies the real smooth factorisation condition $\textnormal{(RFC)}$ in \eqref{FC:Intro:1} with some $k$ and $0\leq k\leq n-1.$ Then we  say that the canonical relation $C$ satisfies the complex factorisation condition (CFC) of rank $k$.     
\end{definition}
  It is important to mention that in the case of a real-valued phase function (CFC) 
is equivalent to the  real factorisation condition (RFC). 

It was proved in \cite{Ruzhansky:CWI-book}, extending the results in \cite{SSS} the following theorem for Fourier integral operators with complex phases.
\begin{theorem}\label{Ruzhansky}
Let $X$ and $Y$ be paracompact manifolds of dimension $n.$ A Fourier integral operator  $T\in  I^\mu_{\rho,1-\rho}(X,Y;C) ,$ $\rho\in [1/2,1],$ with order 
\begin{equation}\label{order:p}
     \mu\leq -({k+(n-k)(1-\rho)})|1/2-1/p|,\,\,
\end{equation}extends to a bounded operator from $L^p_{\textnormal{comp}}(Y)$ into $L^p_{\textnormal{loc}}(X),$ and for $p=1,$ from $H^1_{\textnormal{comp}}(Y)$ into $L^1_{\textnormal{loc}}(X).$    In particular, if $X$ and $Y$ are compact manifolds, the order condition \eqref{order:p} implies that $T$ extends to a bounded operator from $L^p(Y)$ into $L^p(X).$   
\end{theorem}

\subsection{$L^p$-bounds for eigenfunctions of the Laplacian}\label{Sogge:section} Let us consider $(\phi_k,\lambda_k)$ being the spectral data of the operator $\sqrt{1-\Delta_g}$ on a compact Riemannian manifold $(X,g).$ We assume the eigenfunctions $\phi_k$ to be $L^2$-normalised. For our further analysis will be fundamental to have upper bounds for the $L^p$-norms
$$\Vert \phi_k\Vert_{L^p(X)}=\left(\smallint_X|\phi_k(x)|^pdx\right)^{\frac{1}{p}},\,1<p<\infty,\textnormal{ and }\Vert \phi_k\Vert_{L^\infty(X)}=\sup_{x\in X}|\phi_k(x)|.$$
These quantities are a measure of the concentration of $\phi_k$ in $X.$ A celebrated result by Sogge has provided universal upper bounds for $\Vert \phi_k\Vert_{L^p(X)}.$ We record Sogge's result as follows, see \cite{Sogge1988} for details.
\begin{theorem}On a compact Riemmanian manifold $(X,g)$ we have
\begin{equation}
    \Vert \phi_k\Vert_{L^p(X)}\lesssim \lambda_k^{\varphi(p)},\,2\leq p\leq \infty,
\end{equation}where 
\begin{equation}
\varphi(p)= \left\{ \begin{array}{lcc}  n\left(\frac{1}{2}-\frac{1}{p}\right)-\frac{1}{2}  , & \textnormal{if}  & p_c=\frac{2(n+1)}{n-1}\leq p\leq \infty, \\ 
\\ \frac{n-1}{2}\left(\frac{1}{2}-\frac{1}{p}\right) & \textnormal{if} & 2 \leq p\leq p_c . \end{array} \right.
\end{equation}
    
\end{theorem}
The growth of the $L^p$-norms of eigenfunctions has essentially two phases which meet at the critical index $p_c=\frac{2(n+1)}{n-1}.$ Because of the H\"older inequality, for $p$ in the interval $[1,2],$ the $L^p$-norm of an eigenfunction can be estimated from above with its corresponding $L^2$-norm and then for $1<p<2,$ $\varphi(p)=0.$ Then, one can write 
\begin{equation}\label{Sogge:varphi}
\varphi(p)= \left\{ \begin{array}{lcc}  n\left(\frac{1}{2}-\frac{1}{p}\right)-\frac{1}{2}  , & \textnormal{if}  & p_c=\frac{2(n+1)}{n-1}\leq p\leq \infty, \\ 
\\ \frac{n-1}{2}\left(\frac{1}{2}-\frac{1}{p}\right) & \textnormal{if} & 2 \leq p\leq p_c . \\ 
\\ 0 & \textnormal{if} & 1 \leq p\leq 2 .

\end{array} \right.
\end{equation}In our further analysis, we use the function $\varphi(\cdot)$ to describe the orders of Fourier integral operators allowing the membership of these classes on the ideal of $r$-nuclear operators.

\subsection{$r$-Nuclear operators on Banach spaces in the sense of Grothendieck}\label{Nuclearity:section}

 By following the classical reference  Grothendieck \cite{Grothendieck1955}, we  recall that a densely defined linear operator $T:D(T)\subset E\rightarrow F$  (where $D(T)$ is the domain of $T,$ and $E,F$ are Banach spaces) extends to a  $r$-nuclear operator from $E$ to $F$, if
there exist  sequences $(e_n ')_{n\in\mathbb{N}_0}$ in $ E'$ (the dual space of $E$) and $(y_n)_{n\in\mathbb{N}_0}$ in $F$ such that, the discrete representation
\begin{equation}\label{nuc}
Tf=\sum_{n\in\mathbb{N}_0} e_n'(f)y_n,\,\,\, \textnormal{ with }\,\,\,\sum_{n\in\mathbb{N}_0} \Vert e_n' \Vert^r_{E'}\Vert y_n \Vert^r_{F}<\infty,
\end{equation} holds true for all $f\in D(T).$
\noindent The class of $r-$nuclear operators is usually endowed with the natural semi-norm
\begin{equation}
n_r(T):=\inf\left\{ \left\{\sum_n \Vert e_n' \Vert^r_{E'}\Vert y_n \Vert^r_{F}\right\}^{\frac{1}{r}}: T=\sum_n e_n'\otimes y_n \right\}
\end{equation}
\noindent and, if $r=1$, $n_1(\cdot)$ is a norm and we obtain the ideal of nuclear operators. In addition, when $E=F$ is a Hilbert space and  $r=1$ the definition above agrees with that of trace class operators. Moreover,  for the case of Hilbert spaces $H$, the set of $r$-nuclear operators agrees with the Schatten-von Neumann class of order $r,$ with $0<r\leq 1,$ see Oloff \cite{Oloff1970}. 

Let $H$ be a Hilbert space. We recall that for any $r>0,$ a compact operator $T :H\rightarrow H $ belongs to the Schatten von-Neumann ideal $S_{r}(H),$ if the sequence of its singular values $\{s_{n}(T)\}_{n\in \mathbb{N}_0}$ (formed by the eigenvalues of the operator $\sqrt{T^*T}$) belongs to $\ell^r(\mathbb{N}_0),$ that is, if  $$\Vert T\Vert_{S_r}^r:=\sum_{n=1}^\infty s_{n}(T)^r<\infty.$$
For a modern description of other ideals of operators on Hilbert spaces and on Banach spaces we refer the reader to Pietsch  \cite{P,P2}.

In order to determine a relation of the eigenvalues of Fourier integral operators we recall that the nuclear trace of a $r$-nuclear operator on a Banach space coincides with its spectral trace, provided that $0<r\leq \frac{2}{3}.$ This result is due to Grothendieck \cite{Grothendieck1955}.  However, on $L^p(\mu)$-spaces for $\frac{2}{3}\leq r\leq 1$ we present the following extension to Grothendieck's theorem due to Reinov and Latif \cite{Reinov2}). We recall that since the Lebesgue spaces satisfy the {\it approximation property}, the trace of a nuclear operator $T$ as in \eqref{nuc} is well-defined and given by
\begin{equation}\label{Trace:nuclear}
    \textnormal{Tr}(T):=\sum_{j\in\mathbb{N}_0}e_j'(y_j).
\end{equation}
\begin{theorem}\label{Grothendieck:Lidskii:1} Let $T:L^p(\mu)\rightarrow L^p(\mu)$ be a $r$-nuclear operator as in \eqref{nuc} and let $1\leq p\leq \infty$. If $\frac{1}{r}=1+|\frac{1}{p}-\frac{1}{2}|,$ then, 
\begin{equation}
\textnormal{Tr}(T)=\sum_{j\in\mathbb{N}_0}\lambda_j(T)
\end{equation}
where $\lambda_j(T),$ $j\in\mathbb{N}$ is the sequence of eigenvalues of $T$ with multiplicities taken into account. 
\end{theorem} The previous theorem establishes a criterion for the equality of the nuclear trace $\textnormal{Tr}(T)$ with the spectral trace $\sum_{j}\lambda_j(T).$ However, if we want to determine under which conditions the sequence of eigenvalues  $\lambda_j(T),$ $j\in\mathbb{N},$ of a $r$-nuclear operator belongs to some $\ell^q(\mathbb{N}_0)$ space, we have the following result also due to  Reinov and Latif \cite{Reinov1}.
\begin{theorem}\label{Reinov2:q:summability} Let $T:L^p(\mu)\rightarrow L^p(\mu)$ be a $r$-nuclear operator as in \eqref{nuc}. Let $0<r\leq 1,$ $p\in [1,\infty],$ and let $q$ be determined by the equality
\begin{equation}
    \frac{1}{r}=\frac{1}{q}+\left|\frac{1}{p}-\frac{1}{2}\right|.
\end{equation}
Then, 
\begin{equation}
\Vert \{\lambda_{j}(T)\}_{j\in \mathbb{N}_0}\Vert=\left(\sum_{j\in\mathbb{N}_0}|\lambda_j(T)|^{q}\right)^{\frac{1}{q}}\leq C_pn_r(T),
\end{equation}
where $\lambda_n(T),$ $n\in\mathbb{N},$ is the sequence of eigenvalues of $T$ with multiplicities taken into account. 
\end{theorem}

\section{Nuclearity of Fourier integral operators}\label{Nuc:sec}
In this section we investigate the $r$-nuclearity of Fourier integral operators. We are going to use the following notation.
\begin{itemize}
    \item The function $p\mapsto\varphi(p),$ in \eqref{Sogge:varphi} denotes the index of Sogge describing the size of the $L^p$-norm
    $$ \Vert \phi_k\Vert_{L^p(X)}\lesssim \lambda_k^{\varphi(p)},\,1\leq p\leq \infty,$$ for any $L^2$-normalised eigenfunction $\phi_k$ with corresponding eigenvalue $\lambda_k$. We record that
    \begin{equation}\label{Sogge:varphi:2}
\varphi(p)= \left\{ \begin{array}{lcc}  n\left(\frac{1}{2}-\frac{1}{p}\right)-\frac{1}{2}  , & \textnormal{if}  & p_c=\frac{2(n+1)}{n-1}\leq p\leq \infty, \\ 
\\ \frac{n-1}{2}\left(\frac{1}{2}-\frac{1}{p}\right) & \textnormal{if} & 2 \leq p\leq p_c . \\ 
\\ 0 & \textnormal{if} & 1 \leq p\leq 2 .

\end{array} \right.
\end{equation}
    \item For $X=Y,$ the class of Fourier integral operators $I^{\mu}_{\rho,1-\rho}(X,X;C)$ will be denoted by $I^{\mu}_{\rho,1-\rho}(X;C).$  If the canonical relation $C=\Lambda$ is parametrised by a real-valued phase function, the corresponding class of Fourier integral operators will be denoted $I^{\mu}_{\rho,1-\rho}(X,Y;\Lambda).$
    \item The system of eigenvalues of a compact Fourier integral operator $T, $ will be denoted by $\{\lambda_j(T)\}_{j\in \mathbb{N}_0}$ and its sequence of singular values, namely, the eigenvalues of the operator $\sqrt{T^*T}$ will be denoted by $\{s_j(T)\}_{j\in \mathbb{N}_0}.$
\end{itemize}

We start by reviewing the nuclearity of Bessel potentials on Lebesgue spaces.
\subsection{Nuclearity of Bessel potentials on compact manifolds}\label{sect:bessel}
Above we have denoted by  $(\phi_\ell,\lambda_\ell),$ $\ell\in \mathbb{N}_0,$ the spectral data of the operator $\sqrt{1-\Delta_g}$ on a compact Riemannian manifold $(M,g).$ In this section we change such a notation since we have to introduce a suitable terminology involving the dimension $d_\ell$ of each eigenspace $E_\ell:=\textnormal{Ker}(\sqrt{1-\Delta_g}-\lambda_{\ell}).$  Hence,  for each $\ell,$ we are going to consider an orthonormal basis 
$\{e^{k}_{\ell},\,1\leq k\leq d_\ell\}$
of $E_\ell.$ Note that the set $\{e^{k}_{\ell},\,1\leq k\leq d_\ell,\,\ell\in \mathbb{N}_0\}$ provides an orthonormal basis of $L^2(M).$ In the next proposition we investigate the nuclearity of the powers $\sqrt{1-\Delta_g}^{-s}.$
\begin{proposition}\label{Propo:1} Let $0<r\leq 1.$
    Let $(M,g)$ be a closed Riemannian manifold and let $p_1, p_2$ be such that $1\leq p_1<\infty, 1\leq p_2\leq\infty$. Then, the order condition
    \begin{equation}
        s> \frac{n}{r}+\varphi(p_2)+\varphi(p_1'),
    \end{equation}implies that  $E=\sqrt{1-\Delta_g}^{-s}:L^{p_1}(M)\rightarrow L^{p_2}(M)$ is $r$-nuclear.
\end{proposition}
\begin{proof}
  The kernel of $E=\sqrt{1-\Delta_g}^{-s}$ can be written as (see e.g.  \cite[Page 787]{DelgadoRuzhanskyJAM2018})
\begin{equation}
    K(x,y)=\sum_{\ell=0}^{\infty}\sum_{k=1}^{d_\ell} \lambda_\ell^{-s}e_{\ell}^{k}(x)\overline{e_{\ell}^{k}(y)}.
\end{equation} We set 
\begin{equation}
    g_{\ell,k}(x):= \lambda_\ell^{-s}e_{\ell}^{k}(x),\, \quad h_{\ell,k}(y)= \overline{e_{\ell}^{k}(y)}.
\end{equation}Then, note that
\begin{align*}
    \sum_{\ell=0}^{\infty}\sum_{k=1}^{d_\ell}\Vert  g_{\ell,k} \Vert_{L^{p_2}}^r\Vert h_{\ell,k}\Vert_{L^{p_1'}}^r &= \sum_{\ell=0}^{\infty}\sum_{k=1}^{d_\ell} \lambda_\ell^{-sr} \Vert  e_{\ell}^{k} \Vert_{L^{p_2}}^r\Vert e_{\ell}^{k}\Vert_{L^{p_1'}}^r\\
 &\lesssim \sum_{\ell=0}^{\infty}\sum_{k=1}^{d_\ell} \lambda_\ell^{-sr} \lambda_\ell^{\varphi(p_2) r}\lambda_\ell^{\varphi(p_1')r}\\
 &\asymp \sum_{\ell=0}^{\infty}\sum_{k=1}^{d_\ell} \lambda_\ell^{-sr} \lambda_\ell^{\varphi(p_2) r}\lambda_\ell^{\varphi(p_1')r}\\
 &= \sum_{\ell=0}^{\infty}{d_\ell} \lambda_\ell^{-r(s-\varphi(p_2)-\varphi(p_1')}.
\end{align*} Since 
\begin{equation*}
    \sum_{\ell=0}^{\infty}{d_\ell} \lambda_\ell^{-r(s-\varphi(p_2)-\varphi(p_1')}<\infty,
\end{equation*} if and only if
\begin{equation}\label{order:r}
    r(s-\varphi(p_2)-\varphi(p_1')>n,
\end{equation}
 the order condition 
$$s>\frac{n}{r}+\varphi(p_2)+\varphi(p_1')$$ implies that $E:L^{p_1}(M)\rightarrow L^{p_2}(M)$ is $r$-nuclear.  Note that in \eqref{order:r} we have used that for $\kappa\in \mathbb{R},$
\begin{equation*}
    \sum_{\ell=0}^{\infty}{d_\ell} \lambda_\ell^{-\kappa}<\infty,
\end{equation*}if and only if $\kappa>n.$ This is a consequence of the fact that $L=\sqrt{1-\Delta_g}$ is an elliptic pseudo-differential operator of order $\nu=1,$ and of Proposition  5.3 in \cite[Page 779]{DelgadoRuzhanskyJAM2018}.
\end{proof}

\subsection{$L^p$-$L^q$-boundedness of Fourier integral operators with complex phases}\label{lp:lq}

In our further analysis we are going to apply the following equivalent form of the Sobolev embedding theorem, see e.g. \cite{Saloff-CosteBook}.
\begin{remark}[Hardy-Littlewood-Sobolev inequality]\label{HLS:ineq} Let $(X,g)$ be a compact Riemmanian manifold of dimension $n$.
   For any $s\in \mathbb{R},$ let us consider the Bessel potential $\sqrt{1-\Delta_g}^{-s}$ of order $-s.$ Then for $1<p\leq q<\infty,$ the operator 
$$ E=\sqrt{1-\Delta_g}^{-s}:L^p(X)\rightarrow L^q(X),$$ is bounded if $s\geq n \left( \frac{1}{p}-\frac{1}{q}\right)$. For the sharpness of the index $s_0=n \left( \frac{1}{p}-\frac{1}{q}\right)$ we refer to \cite{Saloff-CosteBook}.
\end{remark} 
In the next theorem we present the order conditions for the $L^p$-$L^q$-boundedness of Fourier integral operators from the classes $I^\mu_{\rho,1-\rho}(X,Y;C)$ by combining the methods of \cite[Pages 39-40]{Ruzhansky:CWI-book} and the approach of  H\"ormander \cite{Hor67}.

\begin{proposition} Let $X$ and $Y$ be compact manifolds of dimension $n.$ Let us consider  a Fourier integral operator  $T\in  I^\mu_{\rho,1-\rho}(X,Y;C) ,$ $\rho\in [1/2,1].$ Let us consider the canonical relation $C$ satisfying the complex factorisation condition of rank $k,$ where $0\leq k\leq n-1$.  Then $T$ is bounded from $L^p(Y)$ to $L^q(X)$ if the following order conditions hold:
\begin{itemize}
        \item[(i)] if $1<p\leq q \leq 2$ and  \begin{equation}\label{Necessary:condition:4}
   \mu\leq -n \left( \frac{1}{p}-\frac{1}{q}\right)   -({k+(n-k)(1-\rho)})|1/2-1/q|. \end{equation}    
\item[(ii)] If $2 \leq p \leq q<\infty$ and 

\begin{equation}\label{Necessary:condition:3}
   \mu \leq -n\left( \frac{1}{p}-\frac{1}{q}\right) -({k+(n-k)(1-\rho)})|1/2-1/p|.
   \end{equation}
 \item[(iii)] If  $1<p\leq 2\leq q<\infty$ and 
 \begin{equation}
      \mu \leq -n\left( \frac{1}{p}-\frac{1}{q}\right).
 \end{equation}
\end{itemize}  
\end{proposition}
\begin{proof} According to the position of $p$ and $q$ in the real line, we divide our proof in three cases:
\begin{itemize} 
    \item first, consider $p$ and $q$ with $1<p\leq q \leq 2.$ Consider the Bessel potential on $(Y,g)$ defined via
$$ E=\sqrt{1-\Delta_g}^{-n \left( \frac{1}{p}-\frac{1}{q}\right)}:L^p(Y)\rightarrow L^q(Y),$$ where $g=g_Y$ is the metric on $Y,$ see Remark \ref{HLS:ineq}.  We factorise $T$ with order 
$$ \mu\leq -n \left( \frac{1}{p}-\frac{1}{q}\right)   -({k+(n-k)(1-\rho)})|1/2-1/q|$$
on the right as follows
\begin{equation}
    T=TE^{-1} E.
\end{equation} Note that $TE^{-1}$ is a Fourier integral operator of order 
$$\mu+n \left( \frac{1}{p}-\frac{1}{q}\right) \leq  -({k+(n-k)(1-\rho)})|1/2-1/q|.$$
In consequence $TE^{-1}: L^q(Y)\rightarrow L^q(X)$ is bounded. Then $T:L^p(M)\rightarrow L^q(M)$ is bounded, see Theorem \ref{Ruzhansky}.
\item Similarly, for $p$ and $q$ with  $2 \leq p \leq q<\infty$ and $T$ with order 
$$ \mu\leq -n \left( \frac{1}{p}-\frac{1}{q}\right)   -({k+(n-k)(1-\rho)})|1/2-1/p|,$$
we consider the Bessel potential 
$$ E=\sqrt{1-\Delta_g}^{-n \left( \frac{1}{p}-\frac{1}{q}\right)}:L^p(X)\rightarrow L^q(X),$$ where $g=g_X$ is the metric in $(X,g),$  see Remark \ref{HLS:ineq}. Now, we factorise $T$ on the left as follows
\begin{equation}
    T=E E^{-1}T .
\end{equation} Note that the order of the Fourier integral operator $E^{-1}T$ is 
$$\mu+n \left( \frac{1}{p}-\frac{1}{q}\right) \leq  -({k+(n-k)(1-\rho)})|1/2-1/p|.$$
Hence, $E^{-1}T:L^{p}(Y)\rightarrow L^{p}(X)$ is bounded, see Theorem \ref{Ruzhansky}.  Then $T:L^p(Y)\rightarrow L^q(X)$ is bounded as claimed.
\item  Let us consider $p,q$ satisfying   $1<p\leq 2\leq q<\infty$ and let $T$ to have the order
 \begin{equation}
      \mu \leq -n\left( \frac{1}{p}-\frac{1}{q}\right).
 \end{equation} This case is dramatically different compared with the previous ones. Indeed, consider the Bessel potential on $(Y,g)$ defined via 
$$ E=\sqrt{1-\Delta_{g_Y}}^{-n \left( \frac{1}{p}-\frac{1}{2}\right)}:L^p(Y)\rightarrow L^2(Y),$$ where $g=g_Y$ is the metric on $Y.$ Note that $TE^{-1}$ has order 
\begin{equation}
    \mu+n \left( \frac{1}{p}-\frac{1}{2}\right) \leq -n\left( \frac{1}{p}-\frac{1}{q}\right)+n \left( \frac{1}{p}-\frac{1}{2}\right)= -n\left( \frac{1}{2}-\frac{1}{q}\right).
\end{equation} Now, let us increase the order of this operator by composing it with $F^{-1}$ on the left where 
$$ F=\sqrt{1-\Delta_{g_X}}^{-n \left( \frac{1}{2}-\frac{1}{q}\right)}:L^2(X)\rightarrow L^q(X).$$ The Fourier integral operator $TE^{-1}F^{-1}$ has order zero. Consequently
\begin{equation}
    F^{-1}TE^{-1}:L^2(Y)\rightarrow L^{2}(X),
\end{equation}is bounded. Now, we have the factorisation 
$T=FF^{-1}T E^{-1}E,$ from which we deduce that $T:L^p(Y)\rightarrow L^q(X)$ is bounded as claimed. 
\end{itemize}   The proof is complete. 
\end{proof}

\subsection{Nuclearity of Fourier integral operators with complex phases}\label{Main:nuclearity:section}
In the next theorem we provide a criterion for the $r$-nuclearity of a $L^{p_1}$-$L^{q_{2}}$-bounded Fourier integral operator $T.$ Our strategy is the following. First, we identify the conditions to be satisfied by the order $-s$ of a Bessel potential $E=\sqrt{1-\Delta_g}^{-s}$ in order to have its $r$-nuclearity from $L^{p_1}$ into $L^{p_2},$ and then because of the factorisation $T=TE^{-1}E,$ we identify the order condition on the operator $TE^{-1}$ in order to have its continuity from  $L^{p_2}$ to $L^{q_2}.$ By the structure of these estimates, it is usually required that $p_2\leq q_2$ and then we introduce a parameter $\varkappa\geq 0,$ with
$$q_2=p_2+\varkappa.$$ One reason for this analysis is that in this way we can exploit the $L^{p_2}\rightarrow L^{p_2+\varkappa}$ boundedness of  Fourier integral operators, which under certain cases, e.g. when $p_2\leq 2\leq p_2+\varkappa,$ the $L^{p_2}\rightarrow L^{p_2+\varkappa}$-continuity of a Fourier integral operator has a pseudo-differential behaviour (i.e. with orders that do not depend of the parameter $k,$  and then become independent of the complex factorisation condition).  The independence of the complex factorisation condition in the case $L^{p_2}\rightarrow L^{p_2+\varkappa}$  where $p_2\leq 2\leq p_2+\varkappa,$ is also observed at the level of the $r$-nuclearity of Fourier integral operators.

\begin{theorem}\label{r:nuclearity:FIO} Let $\varkappa\geq 0$ and let $1\leq p_1<  \infty$ and $1<p_2<\infty.$
Let $X$ and $Y$ be compact manifolds of dimension $n.$ Let us assume that $C\subset \widetilde{T^{*}X\times T^{*}Y }$ is a complex canonical relation satisfying the complex factorisation condition of order $k,$ where $0\leq k\leq n-1.$  Let us consider a Fourier integral operator  $T\in  I^\mu_{\rho,1-\rho}(X,Y;C) ,$ where $\rho\in [1/2,1].$ Then $T:L^{p_1}(Y)\rightarrow L^{p_2+\varkappa}(X)$ is $r$-nuclear if the following conditions hold: 
\begin{itemize}
    \item $1<p_2;$ $p_2+\varkappa\leq 2$ and 
    $$ \mu <  -\frac{n\varkappa}{p_2(p_2+\varkappa)}   -({k+(n-k)(1-\rho)})\left|\frac{1}{2}-\frac{1}{p_2+\varkappa}\right|-\frac{n}{r}-\varphi(p_2)-\varphi(p_1').$$
    \item $2\leq p_2$ and 
    $$ \mu <  -\frac{n\varkappa}{p_2(p_2+\varkappa)}   -({k+(n-k)(1-\rho)})\left|\frac{1}{2}-\frac{1}{p_2}\right|-\frac{n}{r}-\varphi(p_2)-\varphi(p_1').$$
    \item $1<p_2\leq 2\leq p_2+\varkappa$ and
        $$ \mu <  -\frac{n\varkappa}{p_2(p_2+\varkappa)}   -\frac{n}{r}-\varphi(p_2)-\varphi(p_1').$$
\end{itemize} 
\end{theorem}

\begin{proof} First, let us define the function $\mu(\varkappa)$ as follows:
$$\mu(\varkappa)= \left\{ \begin{array}{lcc}  -\frac{n\varkappa}{p_2(p_2+\varkappa)}   -({k+(n-k)(1-\rho)})\left|\frac{1}{2}-\frac{1}{p_2+\varkappa}\right|, & \textnormal{if}  & p_2+\varkappa\leq  2 \\ 
\\  -\frac{n\varkappa}{p_2(p_2+\varkappa)}   -({k+(n-k)(1-\rho)})\left|\frac{1}{2}-\frac{1}{p_2}\right| & \textnormal{if} & 2 \leq p_2 \\ 
\\ -\frac{n\varkappa}{p_2(p_2+\varkappa)}  &\textnormal{if} & p_2\leq 2\leq p_2+\varkappa. \end{array} \right. $$
Note that 
\begin{equation}
    \frac{1}{p_2}-\frac{1}{p_2+\varkappa}=\frac{\varkappa}{p_2(p_2+\varkappa)}.
\end{equation}
In consequence, 
    there is $\varepsilon=\varepsilon(\varkappa)>
   0 $ such that
   \begin{equation}
        \mu+\varepsilon= \mu(\varkappa)-\frac{n}{r}-\varphi(p_2)-\varphi(p_1').
   \end{equation} Let us define $$s:= \frac{n}{r}+\varphi(p_2)+\varphi(p_1')+\varepsilon.$$ In view of Proposition \ref{Propo:1}, $E=\sqrt{1-\Delta_g}^{-s}:L^{p_1}(Y)\rightarrow L^{p_2}(Y)$ is $r$-nuclear. Note that we can write $$T=TE^{-1} E$$ with $TE^{-1}$ being an operator of order
   $$ \mu+s =\mu(\varkappa). $$ It follows from Theorem \ref{Ruzhansky} that $TE^{-1}:L^{p_2}(Y)\rightarrow L^{p_2+\varkappa}(X)$ is bounded. In view of the ideal property of the class of $r$-nuclear operators we have that $$T:L^{p_1}(Y)\rightarrow L^{p_2+\varkappa}(X)$$ is $r$-nuclear. 
\end{proof}
Now, motivated by the problem of the distribution of eigenvalues of Fourier integral operators, we consider in Theorem \ref{r:nuclearity:FIO} the case where $p_1=p_2+\varkappa.$ Note that its conjugate exponent is given by 
$$p_1'=\frac{p_2+\varkappa}{p_2+\varkappa-1}.$$ To simplify the notation we drop the sub-index $2$ in $p_2,$ and we present in the next result the $r$-nuclearity conditions for bounded Fourier integral operators on $L^{p+\varkappa}(X).$
\begin{corollary}\label{order:condotions:corollary}
    Let $\varkappa\geq 0$ and let  $1<p<\infty.$
Let $X$ and $Y$ be compact manifolds of dimension $n.$ Let us assume that $C\subset \widetilde{T^{*}X\times T^{*}Y }$ is a complex canonical relation satisfying the complex factorisation condition of order $k,$ where $0\leq k\leq n-1.$ Let us consider a Fourier integral operator  $T\in  I^\mu_{\rho,1-\rho}(X,Y;C) ,$ where $\rho\in [1/2,1].$ Then $T:L^{p+\varkappa}(Y)\rightarrow L^{p+\varkappa}(X)$ is $r$-nuclear if the following conditions hold: 
\begin{itemize}
    \item  $p+\varkappa\leq 2$ and 
    $$ \mu <  \mu(\varkappa)-\frac{n}{r}-\varphi(p)-\varphi\left(\frac{p+\varkappa}{p+\varkappa-1}\right).$$
    \item $2\leq p$ and 
    $$ \mu <  \mu(\varkappa)-\frac{n}{r}-\varphi(p)-\varphi\left(\frac{p+\varkappa}{p+\varkappa-1}\right).$$
    \item $1<p\leq 2\leq p+\varkappa$ and
        $$ \mu < \mu(\varkappa)   -\frac{n}{r}-\varphi(p)-\varphi\left(\frac{p+\varkappa}{p+\varkappa-1}\right),$$
\end{itemize} where we have denoted
\begin{equation}
\mu(\varkappa)= \left\{ \begin{array}{lcc}  -\frac{n\varkappa}{p(p+\varkappa)}   -({k+(n-k)(1-\rho)})\left|\frac{1}{2}-\frac{1}{p+\varkappa}\right|, & \textnormal{if}  & p+\varkappa\leq  2 \\ 
\\  -\frac{n\varkappa}{p(p+\varkappa)}   -({k+(n-k)(1-\rho)})\left|\frac{1}{2}-\frac{1}{p}\right| & \textnormal{if} & 2 \leq p \\ 
\\ -\frac{n\varkappa}{p(p+\varkappa)}  &\textnormal{if} & p\leq 2\leq p+\varkappa. \end{array} \right.
\end{equation}
\end{corollary}
Now, observing that 
\begin{itemize}
    \item with $p=2$ and $\varkappa=0,$ namely, in the case of the $r$-nuclear operators on $L^2(X),$ we are in the Hilbert space setting, and the class of $r$-nuclear operators on $L^2(X),$ for $0<r\leq 1,$ agrees with the Schatten-von Neumann class of order $r,$ $S_r(L^2(X)),$
    \item and that $\mu(0)=0$ for $p=2,$ and also that $\varphi(2)=0.$
\end{itemize} Theorem \ref{r:nuclearity:FIO} has the following consequence. 
\begin{corollary}\label{cor:schatten}
    Let $0<r\leq 1.$
Let $X$ be a compact manifold of dimension $n.$  Let us consider a Fourier integral operator  $T\in  I^\mu_{\rho,1-\rho}(X;C) ,$ where $\rho\in [1/2,1].$ Then $T:L^{2}(X)\rightarrow L^{2}(X)$ belong to the Schatten class of order $r,$ $S_r(L^2(X)),$ provided that $\mu<-n/r.$ 
\end{corollary} A natural question is to determine also the order $\mu$ of a Fourier integral operator in order to guarantee its membership to the ideal $ S_r(L^2(X)),$ in the case $r>1.$ This question is addressed in Proposition \ref{Propo:Schatten}, where we extend Corollary \ref{cor:schatten} to the case $r>1$ giving also an alternative proof for Corollary \ref{cor:schatten} in the range $0<r\leq 1.$  

\subsection{Distribution of eigenvalues of Fourier integral operators on $L^p$-spaces}
In the previous subsection, we have provided order criteria to guarantee the $r$-nuclearity of Fourier integral operators $T:L^{p+\varkappa}(X)\rightarrow L^{p+\varkappa}(X)$. For those operators we have the discrete representation 
\begin{equation}\label{nuc:2T}
Tf=\sum_{j\in\mathbb{N}_0} e_j'(f)y_j,\,\,\, \textnormal{ with }\,\,\,\sum_{j\in\mathbb{N}_0} \Vert e_j' \Vert^r_{E'}\Vert y_j \Vert^r_{E}<\infty, \quad E=L^{p+\varkappa}(X),
\end{equation} from some sequences $\{e_j'\}$ and $\{y_j\}.$ In this case the nuclear trace of the operator $T,$ see \eqref{Trace:nuclear}, is given by
$$\textnormal{Tr}(T)=\sum_{j\in \mathbb{N}_0}e_j'(y_j).$$ 
\begin{remark}\label{ordered:dec}
  According to the notation established above, in what follows the sequence of eigenvalues of $T$ is denoted by $\{\lambda_j(T)\}_{j\in \mathbb{N}_0}.$
Since $T$ is a compact operator, we can assume that the sequence
$\{\vert \lambda_j(T)\vert\}_{j\in \mathbb{N}_0},$ is ordered in non-increasing form
\begin{equation}
    \vert \lambda_0(T)\vert\geq \vert \lambda_1(T)\vert \geq \vert \lambda_2(T)\vert\geq \cdots\geq  \vert \lambda_j(T)\vert\geq  \cdots \geq 0.
\end{equation}  We are going to use this fact in our analysis of the distribution of these eigenvalues.
\end{remark}
In the next Theorem \ref{dist:1}  we provide a rate of decay for the eigenvalues of $r$-nuclear Fourier integral operators. To simplify the notation we are going to use the piecewise  function
\begin{equation}\label{mu:varkappa}
\mu(\varkappa)= \left\{ \begin{array}{lcc}  -\frac{n\varkappa}{p(p+\varkappa)}   -({k+(n-k)(1-\rho)})\left|\frac{1}{2}-\frac{1}{p+\varkappa}\right|, & \textnormal{if}  & p+\varkappa\leq  2 \\ 
\\  -\frac{n\varkappa}{p(p+\varkappa)}   -({k+(n-k)(1-\rho)})\left|\frac{1}{2}-\frac{1}{p}\right| & \textnormal{if} & 2 \leq p \\ 
\\ -\frac{n\varkappa}{p(p+\varkappa)}  &\textnormal{if} & p\leq 2\leq p+\varkappa. \end{array} \right.
\end{equation} The validity of the Grotendieck-Lidskii formula for Fourier integral operators is presented in the following result.
\begin{theorem}\label{dist:1}
    Let $\varkappa\geq 0$ and let  $1<p<\infty.$
Let $X$ be a compact manifold of dimension $n.$ Let us assume that $C\subset \widetilde{T^{*}X\times T^{*}X}$ is a complex canonical relation satisfying the complex factorisation condition of order $k,$ where $0\leq k\leq n-1.$  Let us consider a Fourier integral operator  $T\in  I^\mu_{\rho,1-\rho}(X;C) ,$ where $\rho\in [1/2,1].$  Assume that the order $\mu$ satisfies at least one of the following conditions:
\begin{itemize}
    \item  $p+\varkappa\leq 2$ and 
    $$ \mu <  \mu(\varkappa)-n\left(\frac{1}{2}+\frac{1}{p+\varkappa}\right)-\varphi(p)-\varphi\left(\frac{p+\varkappa}{p+\varkappa-1}\right).$$
    \item $2\leq p$ and 
    $$ \mu <  \mu(\varkappa)-n\left(\frac{3}{2}-\frac{1}{p+\varkappa}\right)-\varphi(p)-\varphi\left(\frac{p+\varkappa}{p+\varkappa-1}\right).$$
    \item $1<p\leq 2\leq p+\varkappa$ and
        $$ \mu < \mu(\varkappa)   -n\left(\frac{3}{2}-\frac{1}{p+\varkappa}\right)-\varphi(p)-\varphi\left(\frac{p+\varkappa}{p+\varkappa-1}\right),$$
\end{itemize} with $\mu(\varkappa)$ as in \eqref{mu:varkappa}. Then $T:L^{p+\varkappa}(X)\rightarrow L^{p+\varkappa}(X)$ is $r$-nuclear with 
$$\frac{1}{r}=1+\left|\frac{1}{2}-\frac{1}{p+\varkappa}\right|$$ and 
\begin{equation}\label{equality:trace}
\textnormal{Tr}(T)=\sum_{j\in\mathbb{N}_0}\lambda_j(T)
\end{equation}
where $\lambda_j(T),$ $j\in\mathbb{N},$ is the sequence of eigenvalues of $T$ with multiplicities taken into account. 
\end{theorem}
\begin{proof} Note that
\begin{itemize}
    \item if $p+\varkappa\leq 2,$
$$-n\left(1+\left|\frac{1}{2}-\frac{1}{p+\varkappa}\right|\right)=-n\left(\frac{1}{2}+\frac{1}{p+\varkappa}\right).$$ 
\item If $p\geq 2,$ or if $p+\varkappa\geq 2,$ we have that 
$$-n\left(1+\left|\frac{1}{2}-\frac{1}{p+\varkappa}\right|\right)=-n\left(\frac{3}{2}-\frac{1}{p+\varkappa}\right).$$     
\end{itemize} In consequence, if we define $r\in (0,1],$ by the identity
\begin{equation*}
    \frac{1}{r}=1+\left|\frac{1}{2}-\frac{1}{p+\varkappa}\right|,
\end{equation*} we have that: 
\begin{itemize}
    \item for $p+\varkappa\leq 2,$
    $$ \mu <  \mu(\varkappa)-\frac{n}{r}-\varphi(p)-\varphi\left(\frac{p+\varkappa}{p+\varkappa-1}\right).$$
    \item With $p\geq 2,$ we have that 
    $$ \mu <  \mu(\varkappa)-\frac{n}{r}-\varphi(p)-\varphi\left(\frac{p+\varkappa}{p+\varkappa-1}\right).$$
    \item For $1<p\leq 2\leq p+\varkappa,$ we have the order inequality
        $$ \mu < \mu(\varkappa)   -\frac{n}{r}-\varphi(p)-\varphi\left(\frac{p+\varkappa}{p+\varkappa-1}\right).$$
\end{itemize}In consequence, Corollary \ref{order:condotions:corollary} implies that $T:L^{p+\varkappa}(Y)\rightarrow L^{p+\varkappa}(X)$ is $r$-nuclear. On the other hand, by applying Theorem \ref{Grothendieck:Lidskii:1} we deduce the validity of the Grothedieck-Lidskii formula \eqref{equality:trace}.
\end{proof} Having investigated the convergence of the sums of eigenvalues in the previous result, we present in the next theorem a criterion for the rate of decay of eigenvalues.

\begin{theorem}\label{dist:2}
    Let $\varkappa\geq 0,$ $0<r\leq 1,$ and let  $1<p<\infty.$
Let $X$ be a compact manifold of dimension $n.$ Let us assume that $C\subset \widetilde{T^{*}X\times T^{*}X }$ is a complex canonical relation satisfying the complex factorisation condition of order $k,$ where $0\leq k\leq n-1.$  Let us consider a Fourier integral operator  $T\in  I^\mu_{\rho,1-\rho}(X;C) ,$ where $\rho\in [1/2,1].$  Assume that the order $\mu$ satisfies at least one of the hypotheses in Corollary \ref{order:condotions:corollary} assuring that $T:L^{p+\varkappa}(X)\rightarrow L^{p+\varkappa}(X)$ is $r$-nuclear. Let  $\lambda_n(T),$ $n\in\mathbb{N},$ be the sequence of eigenvalues of $T$ with multiplicities taken into account.  Then, 
\begin{equation}
\Vert \{\lambda_{j}(T)\}_{j\in \mathbb{N}_0}\Vert_{\ell^q}=\left(\sum_{j\in\mathbb{N}_0}|\lambda_j(T)|^{q}\right)^{\frac{1}{q}}\leq C_pn_r(T),
\end{equation} that is, the sequence $\lambda_n(T),$ $n\in\mathbb{N},$ is $q$-sumable, where $q$ satisfying the inequality $0<q\leq 2,$ is given by 
\begin{equation}
    q=\frac{1}{\frac{1}{r}-\left|\frac{1}{2}-\frac{1}{p+\varkappa}\right|}.
\end{equation} Moreover,  we have the rate of decay
\begin{equation}
    |\lambda_j(T)|=O(j^{ -\frac{1}{r}+\left|\frac{1}{2}-\frac{1}{p+\varkappa}\right|}).
\end{equation}
\end{theorem}
\begin{proof}Note that $p+\varkappa>1.$ Since $0<\frac{1}{p+\varkappa}<1,$ we have that,
either,
$$\frac{1}{p+\varkappa}\leq \frac{1}{2},$$ or that 
$$\frac{1}{p+\varkappa}\geq \frac{1}{2}.$$ In the first case, note that
$$\left|\frac{1}{2}-\frac{1}{p+\varkappa}\right|=\frac{1}{2}-\frac{1}{p+\varkappa}\leq \frac{1}{2}.$$
In the second case, observe that
$$\left|\frac{1}{2}-\frac{1}{p+\varkappa}\right|=\frac{1}{p+\varkappa}-\frac{1}{2}\leq 1- \frac{1}{2}=\frac{1}{2}.$$ 
In consequence
\begin{equation}
   \infty> \frac{1}{r}- \left|\frac{1}{2}-\frac{1}{p+\varkappa}\right|\geq 1-\frac{1}{2}=\frac{1}{2}.
\end{equation} Let us consider 
\begin{equation}
    q=\frac{1}{\frac{1}{r}-\left|\frac{1}{2}-\frac{1}{p+\varkappa}\right|}.
\end{equation}Then $0<q\leq 2,$ and $\frac{1}{q}$ satisfies the identity
$$\frac{1}{r}=\frac{1}{q}+\left|\frac{1}{2}-\frac{1}{p+\varkappa}\right|.$$ In view of Theorem \ref{Reinov2:q:summability}, we have that
\begin{equation}\label{series:q}
\Vert \{\lambda_{j}(T)\}_{j\in \mathbb{N}_0}\Vert=\left(\sum_{j\in\mathbb{N}_0}|\lambda_j(T)|^{q}\right)^{\frac{1}{q}}<\infty,
\end{equation}
where $\lambda_n(T),$ $n\in\mathbb{N},$ is the sequence of eigenvalues of $T$ with multiplicities taken into account. Also, the convergence of this series implies that
\begin{equation}\label{limit}
   \lim_{j\rightarrow\infty} j |\lambda_j(T)|^{q}=0 
\end{equation}
from which we deduce that 
$$|\lambda_j(T)|=O(j^{-\frac{1}{q}}),\quad -\frac{1}{q}=-\frac{1}{r}+\left|\frac{1}{2}-\frac{1}{p+\varkappa}\right|.$$ For the proof of \eqref{limit}, note that the sequence $|\lambda_j(T)|,$ can be assumed to be non-increasing, that is, \begin{equation}
    \vert \lambda_0(T)\vert\geq \vert \lambda_1(T)\vert \geq \vert \lambda_2(T)\vert\geq \cdots\geq  \vert \lambda_j(T)\vert\geq  \cdots \geq 0.
\end{equation} First note that
\begin{equation}\label{eq:1}
    \lim_{k\rightarrow\infty}k\vert \lambda_{2k} \vert^q=0.
\end{equation}Indeed, as the series \eqref{series:q} converges, Cauchy criterion implies that
\begin{equation}
  \lim_{k\rightarrow\infty}k\vert \lambda_{2k} \vert^q= \lim_{k\rightarrow\infty}\sum_{j=1}^k\vert \lambda_{2k} \vert^q\leq \lim_{k\rightarrow\infty} \sum_{j=1}^k \vert \lambda_{k+j} \vert^q =0. 
\end{equation}Note also that
\begin{equation}\label{eq:2}
    \lim_{k\rightarrow\infty}\left(k+\frac{1}{2}\right)\vert \lambda_{2k+1} \vert^q=0.
\end{equation}Indeed, using a similar analysis we have that 
\begin{align*}
    \lim_{k\rightarrow\infty}\left(k+\frac{1}{2}\right)\vert \lambda_{2k+1} \vert^q &=\lim_{k\rightarrow\infty} \frac{1}{2}\vert \lambda_{2k+1} \vert^q+k\vert \lambda_{2k+1} \vert^q\\
    &= \lim_{k\rightarrow\infty}  \frac{1}{2}\vert \lambda_{2k+1} \vert^q+\sum_{j=1}^{k}\vert \lambda_{2k+1} \vert^q\\
    &\leq \lim_{k\rightarrow\infty}  \frac{1}{2}\vert \lambda_{2k+1} \vert^q+\sum_{j=0}^{k}\vert \lambda_{k+j+1} \vert^q=0.
\end{align*}Multiplying both sides of \eqref{eq:1} and of \eqref{eq:2} by $2,$ we get \eqref{limit}.
The proof is complete.
\end{proof}
\begin{remark}
    Note that when $p+\varkappa=2,$ (see Corollary \ref{cor:schatten}), the order condition $\mu<-n/r,$ $0<r\leq 1,$ implies that $T:L^2(X)\rightarrow L^2(X)$ is $r$-nuclear, or equivalently, it belongs to the Schatten-von Neumann class $S_r(L^2(X))$. By applying Theorem \ref{dist:2}, we have that the system of eigenvalues of $T,$ obeys the rate of decay
    \begin{equation}\label{nuclear:vs:weyl}
        |\lambda_j(T)|=O(j^{-\frac{1}{r}}).
    \end{equation}
\end{remark}Using a classical inequality by Hermann Weyl \cite{Weyl1949}, we are going to extend \eqref{nuclear:vs:weyl} to the case $r>1$ in Subsection \ref{FIO:dist}.

\section{Schatten-von Neumann properties of Fourier integral operators revisited}\label{Schatten:sec} 

\subsection{Sufficient conditions}
On a Hilbert space $H$, the ideal of $r$-nuclear operators agrees with the Schatten-von Neumann class of order $r,$ for $0<r\leq 1.$ We recall that on general Banach spaces for $r>1,$ the ideal of $r$-nuclear operators does not increase the elements of the case $r\leq 1$. Having investigated the $r$-nuclearity of Fourier integral operators in the previous section, we are now going to investigate the Hilbert space case, namely when $H=L^{2}(X),$ and then  criteria will be given on the order of these operators to guarantee their membership to the Schatten-von Neumann classes $S_{r}(L^2(X)).$ With this goal, in the next proposition we present some sufficient conditions.

\begin{proposition}\label{Propo:Schatten} Let $0<r<\infty.$
Let $X$ be a compact manifold of dimension $n.$  Let us consider a Fourier integral operator  $T\in  I^\mu_{\rho,1-\rho}(X;C) ,$ where $\rho\in [1/2,1].$ Then $T:L^{2}(X)\rightarrow L^{2}(X)$ belong to the Schatten class of order $r,$ $S_r(L^2(X)),$ provided that $\mu<-n/r.$    
\end{proposition}
\begin{proof} Note that $T\in S_r(L^2(X))$ if and only if $T^*T\in S_{\frac{r}{2}}(L^2(X)). $ Indeed, note that
\begin{equation}\label{equi:r:r:2}
 \Vert T\Vert_{S_r}^{r}=\sum_{k=1}^{\infty}s_{k}(\vert T\vert)^r=  \sum_{k=1}^{\infty}s_{k}(\vert T\vert^2)^{\frac{r}{2}} =\sum_{k=1}^{\infty}s_{k}(T^{*}T)^{\frac{r}{2}}=\Vert T^*T\Vert_{S_{\frac{r}{2}}}^{\frac{r}{2}}.
\end{equation} Note that $ T^*T$ is a pseudo-differential operator with principal symbol $\vert a\vert^2$ if $a$ is the principal symbol of $T.$ Moreover, $T^*T\in \Psi^{2\mu}_{\rho,1-\rho}(X).$ Then, $T^*T\in  S_{\frac{r}{2}}(L^2(X))$ if 
$$ 2\mu<-\frac{n}{r/2},$$ which is equivalent to the order inequality $\mu<-n/r.$ To prove that $$T^*T\in  S_{\frac{r}{2}}(L^2(X))$$ for $\mu<-\frac{n}{r},$ just notice that we can factorise the operator $T^*T$ as follows
$$T^*T=T^*T\sqrt{1+\Delta}_{g}^{-2\mu} \sqrt{1+\Delta}_{g}^{2\mu}.$$ Since $$T^*T\sqrt{1+\Delta}_{g}^{-2\mu}\in \Psi^{0}_{\rho,1-\rho}(X).$$ By the Calder\'on-Vaillancourt theorem we have that $T^*T\sqrt{1+\Delta}_{g}^{-2\mu}$ is bounded on $L^2(X).$ At the same time $$\sqrt{1+\Delta}_{g}^{2\mu}\in S_{\frac{r}{2}}(L^2(X)),$$
if and only if $2\mu<-\frac{n}{r/2.}$ By the ideal property of the Schatten-classes we have that for $\mu<-n/r,$ one has that $T^*T\in S_{\frac{r}{2}}(L^2(X)), $ or equivalently, $T\in S_r(L^2(X)).$ The proof is complete.    
\end{proof}

\subsection{Necessary conditions for pseudo-differential operators}\label{Necess:Cond:sec}
Now, we are going to analyse the sharpness of the order $\mu<-n/r,$ in Proposition \ref{Propo:Schatten} in the case of pseudo-differential operators, namely the case where
$$\Phi(x,y,\theta):=(x-y,\theta).$$
For this, we will make use of the local Weyl formula for pseudo-differential operators, see Zelditch \cite{Zelditch1994}. 

\begin{theorem}[Local Weyl formula for pseudo-differential operators]\label{T:LWF2:PSDO} Let  $F\in \Psi^{0}_{1,0}(X),$ be pseudo-differential operator of order zero. Then
\begin{equation}\label{LWF2:PSDO}
    \lim_{\lambda\rightarrow\infty}\frac{1}{N(\lambda)}\sum_{k: \lambda_k\leq \lambda}(F\phi_k,\phi_k)=\smallint_{T^*\mathbb{S}(X)}\sigma_F d\mu_{L},
\end{equation} where $T^*\mathbb{S}(X)$ denotes the co-sphere bundle  and  $d\mu_L(x,\xi)$ is the usual
Liouville measure. 
    
\end{theorem}  
Taking inspiration from the methods of harmonic analysis, we are going to estimate the sum $\sum_{\frac{\lambda}{2}<\lambda_k\leq \lambda}(F\phi_k,\phi_k),$ for our further analysis.

\begin{lemma}\label{dyadic:lemma:PSDO} Let $F\in \Psi^{0}_{1,0}(X)$ be a pseudo-differential operator such that the sequence
$$\{(F\phi_k,\phi_k): k\in \mathbb{N}_0\}$$ is real-valued\footnote{This condition is satisfied if e.g. $F$ is self-adjoint.}.
Then, for any $\varepsilon>0,$ there exists $\lambda_0({\varepsilon})>0,$ such that for $\lambda\geq 2\lambda_0(\varepsilon),$ one has that
$$\sum_{\frac{\lambda}{2}<\lambda_k\leq \lambda}(F\phi_k,\phi_k)$$ 
\begin{equation}
  \geq C_n\lambda^n\left(\left(1-\frac{1}{2^n}\right)\smallint_{T^*\mathbb{S}(X)}\sigma_F d\mu_{L}-\varepsilon\left(1+\frac{1}{2^n}\right) \right)+O(\lambda^{n-1}),
\end{equation}where $C_n$ is a geometric constant.
\end{lemma}
\begin{proof}
     Note that the fact that $(F\phi_k,\phi_k)$ is real-valued assures that the right-hand side of \eqref{LWF2:PSDO} is real-valued.  Let us fix $\varepsilon>0.$ There exists $\lambda_0(\varepsilon)$ such that for any $\lambda> \lambda_0(\varepsilon)$ we have that
     \begin{equation}
         -\varepsilon+\smallint_S \sigma_F d\mu_{L}\leq \frac{1}{N(\lambda)}\sum_{k: \lambda_k\leq \lambda}(F\phi_k,\phi_k)\leq \varepsilon+\smallint_S \sigma_F d\mu_{L},
     \end{equation}
where we have written
$$\smallint_S \sigma_F d\mu_{L} = \smallint_{T^*\mathbb{S}(X)}\sigma_F d\mu_{L} $$ by denoting $S=T^*\mathbb{S}(X)$  to simplify the notation. At this point we recall the Weyl formula
$$ N(\lambda)=C_n\lambda^n+O(\lambda^{n-1}) .$$
In particular, if $\lambda>2\lambda_0(\varepsilon)$ we have the inequality
\begin{equation}
    \frac{1}{N(\lambda/2)}\sum_{k: \lambda_k\leq \lambda/2}(F\phi_k,\phi_k)\leq \varepsilon+\smallint_S \sigma_Fd\mu_L,
\end{equation} from which we have that 
$$ -\sum_{k: \lambda_k\leq \lambda/2}(F\phi_k,\phi_k)\geq -N(\lambda/2)\left(\varepsilon+\smallint_S \sigma_Fd\mu_L\right) $$
$$=-\left( \frac{C_n}{2^n}\lambda^n+\frac{1}{2^{n-1}}O(\frac{1}{2^{n-1}}\lambda^{n-1})\right)\left(\varepsilon+\smallint_S \sigma_Fd\mu_L\right).$$ On the other hand note that the inequality
$$ \frac{1}{N(\lambda)}\sum_{k: \lambda_k\leq \lambda}(F\phi_k,\phi_k) \geq -\varepsilon+\smallint_S \sigma_F d\mu_{L}  $$
can be written as follows
\begin{equation}
   \sum_{k: \lambda_k\leq \lambda}(F\phi_k,\phi_k)\geq \left(C_n\lambda^n+O(\lambda^{n-1})\right) \left( -\varepsilon+\smallint_S \sigma_F d\mu_{L}\right).
\end{equation} Now let us estimate from below the contribution of the terms $(T\phi_k,\phi_k)$ in the region
$$\left\{\lambda_k: {\lambda}/{2}<\lambda_k\leq \lambda\right\}.$$ Note that
\begin{align*}
    \sum_{k:\frac{\lambda}{2}<\lambda_k\leq \lambda}(F\phi_k,\phi_k) &= \sum_{k: \lambda_k\leq \lambda}(F\phi_k,\phi_k)-\sum_{k: \lambda_k\leq \lambda/2}(F\phi_k,\phi_k)\\
    &\geq \left(C_n\lambda^n+O(\lambda^{n-1})\right) \left( -\varepsilon+\smallint_S \sigma_F d\mu_{L}\right)\\
    &-\left( \frac{C_n}{2^n}\lambda^n+\frac{1}{2^{n-1}}O(\frac{1}{2^{n-1}}\lambda^{n-1})\right)\left(\varepsilon+\smallint_S \sigma_Fd\mu_L\right)\\
    &=C_n\lambda^{n}\left( -\varepsilon+\smallint_S \sigma_F d\mu_{L}\right)+O(\lambda^{n-1})\left( -\varepsilon+\smallint_S \sigma_F d\mu_{L}\right)\\
    &-\frac{C_n}{2^n}\lambda^n\left(\varepsilon+\smallint_S \sigma_Fd\mu_L\right)+O(\frac{1}{2^{n-1}}\lambda^{n-1})\left(\varepsilon+\smallint_S \sigma_Fd\mu_L\right)\\
    &=C_n\lambda^{n}\left( -\varepsilon+\smallint_S \sigma_F d\mu_{L}\right)-\frac{C_n}{2^n}\lambda^n\left(\varepsilon+\smallint_S \sigma_Fd\mu_L\right)
    +O(\lambda^{n-1})\\
    &=C_{n}\lambda^n\left(\smallint_S \sigma_F d\mu_{L}\left(1-\frac{1}{2^n}\right)-\varepsilon\left(1+\frac{1}{2^n}\right)\right)+O(\lambda^{n-1}).
\end{align*}The proof is complete.     
\end{proof}

\begin{proposition}\label{propo:PSDO} 
Let $X$ be a compact manifold of dimension $n.$  Let us consider a pseudo-differential operator  $T\in  \Psi^\mu_{1,0}(X) .$ Then $T:L^{2}(X)\rightarrow L^{2}(X)$ belongs to the Schatten class $S_1(L^2(X)),$  namely, $T$ is of trace class on $L^2,$ if and only if  $\mu <-n,$ provided that at least one of the following averages 
\begin{equation}\label{positivity}
    \smallint_{T^*\mathbb{S}(X)}\textnormal{Re}(\sigma_F) d\mu_{L}, \quad \smallint_{T^*\mathbb{S}(X)}\textnormal{Im}(\sigma_F)  d\mu_{L},
\end{equation}
is strictly positive.    
\end{proposition}
\begin{proof}
    In view of Proposition \ref{Propo:Schatten}, the order condition $\mu<-n,$ implies that $T\in S_1(L^2(X))$ (i.e. $T$ is of trace class on $L^{2}(X)$). Now, let us consider $T$ to be of trace class. Since $$\textnormal{Re}(T)=\frac{T+T^*}{2}$$ and $$\textnormal{Im}(T)=\frac{T-T^*}{2i},$$ are also of trace class we have that 
    \begin{equation}
        \textnormal{Tr}(T)= \textnormal{Tr}(\textnormal{Re}(T))+i\textnormal{Tr}(\textnormal{Im}(T)).
    \end{equation} Let $T_0=\textnormal{Re}(T),$ and $T_1=\textnormal{Im}(T).$ Note that each $T_\ell$ is self-adjoint on $L^{2}(X),$ where $\ell=0,1.$ Let us estimate the trace $\textnormal{Tr}(T_\ell).$ Indeed,  if $(\phi_k,\lambda_k)$ are the spectral data of $\sqrt{1-\Delta_g}$ 
we have that 
\begin{align*}
    \textnormal{Tr}(T_\ell)=\sum_{\lambda_k}(T_\ell\phi_k,\phi_k) &\asymp \sum_{s=0}^{\infty}\sum_{2^{s-1}<\lambda_k\leq 2^{s}  } (T_\ell\phi_k,\phi_k)\\
    &= \sum_{s=0}^{\infty}\sum_{2^{s-1}<\lambda_k\leq 2^{s}  } \lambda_k^{\mu}(T_\ell \lambda_k^{-\mu}\phi_k,\phi_k)\\
     &= \sum_{s=0}^{\infty}\sum_{2^{s-1}<\lambda_k\leq 2^{s}  } \lambda_k^{\mu}(T_\ell \sqrt{1-\Delta_g})^{-\mu}\phi_k,\phi_k).
\end{align*}
Without loss of generality we can assume that 
$$\smallint_{T^*\mathbb{S}(X)}\textnormal{Re}(\sigma_F) d\mu_{L}>0.$$
We can choose $\varepsilon>0$ such that 
$$ C_\varepsilon=\left(1-\frac{1}{2^n}\right)\smallint_{T^*\mathbb{S}(X)}\textnormal{Re}(\sigma_F) d\mu_{L}-\varepsilon\left(1+\frac{1}{2^n}\right)>0.$$ Indeed, such a $\varepsilon$ satisfies the inequality
$$0<\varepsilon< \frac{\left(1-\frac{1}{2^n}\right)}{\left(1+\frac{1}{2^n}\right)}\smallint_{T^*\mathbb{S}(X)}\textnormal{Re}(\sigma_F).$$
Note that, $\textnormal{Re}(\sigma_F) $ is the principal symbol of $T_0,$ and that the sequence
$$ (T_0 \sqrt{1-\Delta_g})^{-\mu}\phi_k,\phi_k)= \lambda_k^{-\mu}(T_0 \phi_k,\phi_k)$$ is real-valued. Observe also that the pseudo-differential operator $$F_{0}=T_0\sqrt{1-\Delta_g})^{-\mu} \in I^0_{1,0}(X)$$ has order zero. In view of  Lemma \ref{dyadic:lemma:PSDO}, there exists $\lambda_0(\varepsilon)$ such that for $\lambda\geq \lambda_0(\varepsilon), $ one has that 
$$\sum_{2^{s-1}<\lambda_k\leq 2^{s}}(F_0\phi_k,\phi_k)\geq  C_nC_\varepsilon 2^{sn}+O(2^{s(n-1)})$$
where $C_n$ is a geometric constant. Note that 

\begin{align*}
     \textnormal{Tr}(T_0) &\asymp \sum_{s=0}^{\infty}\sum_{2^{s-1}<\lambda_k\leq 2^{s}  } \lambda_k^{\mu}(T_0 \sqrt{1-\Delta_g})^{-\mu}\phi_k,\phi_k)\\
     &\asymp \sum_{s=0}^{\infty}\sum_{2^{s-1}<\lambda_k\leq 2^{s}  } 2^{s\mu}(T_0 \sqrt{1-\Delta_g})^{-\mu}\phi_k,\phi_k)\\
     &\gg_{\varepsilon} \sum_{s=0}^{\infty} (C_nC_\varepsilon 2^{s(n+\mu)}+O(2^{s(\mu+n-1)})).
\end{align*} If the order $\mu$ is such that  $\mu\geq -n,$ the  series $$ \sum_{s=0}^{\infty} (C_nC_\varepsilon 2^{s(n+\mu)}+O(2^{s(\mu+n-1)})),$$ diverges which contradicts that $T_0$ is of trace class. In consequence, we have that $\mu<-n.$ The proof is complete.  
\end{proof}

\begin{theorem}\label{th:trace:class:PSDO} Let $X$ be a compact manifold of dimension $n.$  Let us consider a pseudo-differential operator  $T\in  \Psi^\mu_{1,0}(X) .$ Then $T:L^{2}(X)\rightarrow L^{2}(X)$ belongs to the Schatten class $S_1(L^2(X)),$  namely, $T$ is of trace class on $L^2,$ if and only if  $\mu <-n.$ 
    
\end{theorem}
\begin{proof} In view of Proposition \ref{Propo:Schatten}, the order condition $\mu<-n,$ implies that $T\in S_1(L^2(X))$ (i.e. $T$ is of trace class on $L^{2}(X)$). On the other hand, if $T$ is a trace class operator and if at least one of the following averages 
\begin{equation}
    \smallint_{T^*\mathbb{S}(X)}\textnormal{Re}(\sigma_F) d\mu_{L}, \quad \smallint_{T^*\mathbb{S}(X)}\textnormal{Im}(\sigma_F)  d\mu_{L},
\end{equation}
is strictly positive, the result has been proved in Proposition \ref{propo:PSDO}. In the general case, assume that $T$ is of trace class. Then $\textnormal{Re}(T)=\frac{1}{2}(T+T^*)$ is of trace class. 

Since $S=\textnormal{Re}(T)$ is a self-adjoint operator, note that $S^*=\textnormal{Re}(T).$  Moreover
$$  \textnormal{Re}(T)^2=S^2=S^*S\in \Psi^{2\mu}_{1,0}(X),$$ is a pseudo-differential operator of order $2\mu.$  In consequence, the composition operator $\textnormal{Re}(T)^2\sqrt{1-\Delta_g}^{-\mu}$ is also a pseudo-differential operator. Therefore, we have that
$$ \textnormal{Re}(T)^2\sqrt{1-\Delta_g}^{-\mu}= \textnormal{Re}(T)\textnormal{Re}(T)\sqrt{1-\Delta_g}^{-\mu}\in \Psi^{\mu}_{1,0}(X).$$
Let us assume that $T\neq 0,$ has  principal symbol $\sigma_T,$ with $\textnormal{Re}(\sigma_T)\not\equiv 0.$
Let $\varkappa>0$  and define the operator
\begin{align*}
 F=  \textnormal{Re}(T)+\varkappa  \textnormal{Re}(T)^2\sqrt{1-\Delta_g}^{-\mu} \in \Psi^{\mu}_{1,0}(X).
\end{align*} The principal symbol of $F$ is real-valued and can be computed as follows
$$\sigma_F=\textnormal{Re}(\sigma_T)+\varkappa \textnormal{Re}(\sigma_T)^2(1+\Vert\xi\Vert)^{-\mu}=\textnormal{Re}(\sigma_F).$$ On the other hand
by choosing $\varkappa>0$ such that 
\begin{equation*}
  \varkappa>-\frac{\smallint_{T^*\mathbb{S}(X)}\textnormal{Re}(\sigma_T) d\mu_{L}}{\smallint_{T^*\mathbb{S}(X)}\textnormal{Re}(\sigma_T)^2(1+\Vert\xi\Vert)^{-\mu} d\mu_{L}}  
\end{equation*} we have that 
\begin{align*} &\smallint_{T^*\mathbb{S}(X)}\textnormal{Re}(\sigma_F) d\mu_{L}\\
&=\smallint_{T^*\mathbb{S}(X)}\textnormal{Re}(\sigma_T) d\mu_{L}+\varkappa \smallint_{T^*\mathbb{S}(X)} \textnormal{Re}(\sigma_T)^2(1+\Vert\xi\Vert)^{-\mu} d\mu_{L}>0.    
\end{align*}So, the pseudo-differential operator $ F=  \textnormal{Re}(T)+\varkappa  \textnormal{Re}(T)^2\sqrt{1-\Delta_g}^{-\mu} \in \Psi^{\mu}_{1,0}(X)$ has order $\mu,$ is of trace class and its (real-valued) principal symbol $\sigma_F$ satisfies the positivity condition 
$$\smallint_{T^*\mathbb{S}(X)}\textnormal{Re}(\sigma_F) d\mu_{L}>0.$$ It follows from Proposition \ref{propo} that $\mu<-n.$  On the other hand, if $\textnormal{Re}(\sigma_T)\equiv 0,$ since $T\neq 0,$ we have that  $\textnormal{Im}(\sigma_T)\not\equiv  0.$ Then, we repeat all the analysis above by considering 
\begin{align*}
 F=  \textnormal{Im}(T)+\varkappa  \textnormal{Im}(T)^2\sqrt{1-\Delta_g}^{-\mu} \in \Psi^{\mu}_{1,0}(X).
\end{align*}  By choosing    
$\varkappa>0$ such that 
\begin{equation*}
  \varkappa>-\frac{\smallint_{T^*\mathbb{S}(X)}\textnormal{Im}(\sigma_T) d\mu_{L}}{\smallint_{T^*\mathbb{S}(X)}\textnormal{Im}(\sigma_T)^2(1+\Vert\xi\Vert)^{-\mu} d\mu_{L}}  
\end{equation*} we have that 
\begin{align*} &\smallint_{T^*\mathbb{S}(X)}\textnormal{Im}(\sigma_F) d\mu_{L}\\
&=\smallint_{T^*\mathbb{S}(X)}\textnormal{Im}(\sigma_T) d\mu_{L}+\varkappa \smallint_{T^*\mathbb{S}(X)} \textnormal{Im}(\sigma_T)^2(1+\Vert\xi\Vert)^{-\mu} d\mu_{L}>0.    
\end{align*} Then, the pseudo-differential operator $ F=  \textnormal{Im}(T)+\varkappa  \textnormal{Im}(T)^2\sqrt{1-\Delta_g}^{-\mu} \in \Psi^{\mu}_{1,0}(X)$ has order $\mu,$ is of trace class and its (real-valued) principal symbol $\sigma_F$ satisfies the positivity condition 
$$\smallint_{T^*\mathbb{S}(X)}\textnormal{Im}(\sigma_F) d\mu_{L}>0.$$ It follows from Proposition \ref{propo} that $\mu<-n.$
The proof is complete. 
\end{proof}

\begin{lemma}\label{Invariance:trace:2:PSDO} Let $X$ be a compact manifold of dimension $n.$  Let us consider a pseudo-differential operator  $A\in  \Psi^\mu_{1,0}(X) $ and let $r>1.$  If $A\in S_r(L^2(X))$ then $\mu\leq -n/r.$
\end{lemma}
\begin{proof}
    Let $\varepsilon>0.$ Consider the operator $\sqrt{1-\Delta}^s$ where $s$ satisfies
$$ s:=-\frac{n}{q}-\varepsilon<-\frac{n}{q},$$ and where $$q:=r/(r-1)$$ is the conjugate exponent of $r.$ Then, $q$ is given by the identity 
$1=\frac{1}{q}+\frac{1}{r}.$ Note that since $r>1,$ one has that $q>1.$ Note also that $$\sqrt{1-\Delta}^s\in S_q(L^2(X)).$$ Since $q$ and $r$ and conjugate exponents, we have that
\begin{equation*}
    \tilde{A}:=A\sqrt{1-\Delta}^s\in S_1(L^2(X)).
\end{equation*}By Theorem \ref{th:trace:class:PSDO} we have that the order $\mu+s$ of $\tilde A\in S_1$ satisfies the inequality 
$\mu+s<-n.$ Then, we have that 
$$  \mu<-s-n=\frac{n}{q}+\varepsilon -n=n\left(\frac{1}{q}-1\right)+\varepsilon =-\frac{n}{r}+\varepsilon\,, $$
where the last hold true for any $\epsilon>0$.
Taking $\varepsilon\rightarrow 0^+$ we have that $\mu\leq -n/r.$
\end{proof}

\begin{theorem}\label{even:condition:PDSO} Let $X$ be a compact manifold of dimension $n.$  Let us consider a pseudo-differential operator  $T\in  I^\mu_{1,0}(X) $ and let $r\in \mathbb{N}= (0,\infty)\cap \mathbb{Z}.$ Then $T:L^{2}(X)\rightarrow L^{2}(X)$ belongs to the Schatten class $S_r(L^2(X)),$  if and only if  $\mu <-n/r.$ 
    
\end{theorem}
\begin{proof} In view of Proposition \ref{Propo:Schatten}, the order condition  $\mu<-n/r,$ $0<r<\infty,$ implies that $T\in S_r(L^2(X)).$ Now, let us prove the converse statement, that is if $T\in S_r(L^2(X))$, then $\mu<-n/r.$  Now, we assume that $r\geq 1$ and that $r\in \mathbb{Z}.$
To prove the theorem let us show that the borderline $\mu=-n/r$ in Lemma \ref{Invariance:trace:2:PSDO} is not possible.  Since $T\in S_r(L^2(X)),$ we have  that
\begin{equation}
    T^r=T\circ T\circ\cdots T\in S_{r}(L^2(X))\circ\cdots\circ S_{r}(L^2(X))\subseteq S_{1}(L^2(X)),
\end{equation} where the composition $\circ$ is taken $r$-times. However, the order of the trace class operator $T^r$ is $$\mu r=-n$$ and this contradicts the conclusion in Theorem \ref{th:trace:class:PSDO}. So, necessarily $\mu<-n/r.$     
\end{proof}

\subsection{Necessary conditions for Fourier integral operators}\label{Necess:Cond:sec}
Now, we are going to analyse the sharpness of the order $\mu<-n/r,$ in Proposition \ref{Propo:Schatten}. Our main result in this section is Theorem \ref{th:trace:class} which we can prove making use of Theorem \ref{th:trace:class:PSDO}.
At the same time, we also prove that Theorem \ref{th:trace:class} is a consequence of the local Weyl formula for Fourier integral operators when $C=\Lambda$ is a real canonical relation. Both approaches will be presented in the proof of Theorem \ref{th:trace:class}. 
Next, we present the local Weyl formula due to Zelditch and Kuznecov \cite{Zelditch:Kuznecov:1992}. Here, $d\mu_L(x,\xi)$ is the usual
Liouville measure. 

\begin{theorem}[Local Weyl formula for Fourier integral operators]\label{T:LWF} Let $\Lambda$ be a real canonical relation and let $F\in I^{0}_{1,0}(X,\Lambda).$ Then
\begin{equation}\label{LWF}
    \lim_{\lambda\rightarrow\infty}\frac{1}{N(\lambda)}\sum_{k: \lambda_k\leq \lambda}(F\phi_k,\phi_k)=\smallint_{S(\Lambda \cap \Delta_{S^*X\times S^*X } )}\sigma_F d\mu_{L},
\end{equation} where $S(\Lambda \cap \Delta_{S^*X\times S^*X } )$ is the set of unit vectors in the diagonal part of $\Lambda.$
    
\end{theorem}  
The major part of this subsection is dedicated to constructing the proof of the following fact: the Local Weyl formula for Fourier integral operators implies the characterisation of trace class Fourier integral operators in Theorem \ref{th:trace:class}.
The local Weyl formula implies the following result.
\begin{corollary}
     Let $\Lambda$ be a real canonical relation and let $F\in I^{0}_{1,0}(X,\Lambda).$ Then the real part of $F$ and the imaginary part of $F,$ satisfy
\begin{equation}\label{LWF:RePart}
    \lim_{\lambda\rightarrow\infty}\frac{1}{N(\lambda)}\sum_{k: \lambda_k\leq \lambda}(\textnormal{Re}(F)\phi_k,\phi_k)=\smallint_{S(\Lambda \cap \Delta_{S^*X\times S^*X } )}\textnormal{Re}(\sigma_F) d\mu_{L},
\end{equation}
and 
\begin{equation}\label{LWF:ImPart}
    \lim_{\lambda\rightarrow\infty}\frac{1}{N(\lambda)}\sum_{k: \lambda_k\leq \lambda}(\textnormal{Im}(F)\phi_k,\phi_k)=\smallint_{S(\Lambda \cap \Delta_{S^*X\times S^*X } )}\textnormal{Im}(\sigma_F) d\mu_{L},
\end{equation}
where $S(\Lambda \cap \Delta_{S^*X\times S^*X } )$ is the set of unit vectors in the diagonal part of $\Lambda.$
    
\end{corollary}
\begin{proof}
    The proof of this corollary can be obtained by taking the real part and the imaginary part of \eqref{LWF} in both sides, and using the continuity of these functions one obtain \eqref{LWF:RePart} and \eqref{LWF:ImPart}, respectively. 
\end{proof}

In the next lemma, we are going to estimate the sum $\sum_{\frac{\lambda}{2}<\lambda_k\leq \lambda}(F\phi_k,\phi_k),$ where $F$ is the real or the imaginary part of a Fourier integral operator.

\begin{lemma}\label{dyadic:lemma} Let $\Lambda$ be a real canonical relation and let $S\in I^{0}_{1,0}(X,\Lambda).$ Let us consider
$$F_0:=\textnormal{Re}(S)=\frac{S+S^*}{2}$$ and $$F_1:=\textnormal{Im}(S)=\frac{S-S^*}{2i},$$ being the real and the imaginary part of $S,$ respectively.
Then, for any $\varepsilon>0,$ there exists $\lambda_0({\varepsilon})>0,$ such that for $\lambda\geq 2\lambda_0(\varepsilon),$ one has that
$$\sum_{\frac{\lambda}{2}<\lambda_k\leq \lambda}(F_0\phi_k,\phi_k)$$ 
\begin{equation}
  \geq C_n\lambda^n\left(\left(1-\frac{1}{2^n}\right)\smallint_{S(\Lambda \cap \Delta_{S^*X\times S^*X } )}\textnormal{Re}(\sigma_{S}) d\mu_{L}-\varepsilon\left(1+\frac{1}{2^n}\right) \right)+O(\lambda^{n-1}),
\end{equation}
and also
$$\sum_{\frac{\lambda}{2}<\lambda_k\leq \lambda}(F_1\phi_k,\phi_k)$$ 
\begin{equation}
  \geq C_n\lambda^n\left(\left(1-\frac{1}{2^n}\right)\smallint_{S(\Lambda \cap \Delta_{S^*X\times S^*X } )}\textnormal{Im}(\sigma_{S}) d\mu_{L}-\varepsilon\left(1+\frac{1}{2^n}\right) \right)+O(\lambda^{n-1}),
\end{equation}
where $C_n$ is a geometric constant, and $\ell=0,1.$
\end{lemma}
\begin{proof} We proceed with the proof with case of $F=F_0,$ and by abuse of notation let us write
$$\sigma_F=\textnormal{Re}(\sigma_{S}).$$ At this point we think of $\sigma_F$ as a function and not as the symbol of a Fourier integral operator. Indeed, $\textnormal{Re}(F)$ is the sum of two Fourier integral operators with different canonical relations.  
     Note that the fact that $F$ is self-adjoint implies that $(F\phi_k,\phi_k)$ is real-valued and then the right-hand side of \eqref{LWF} is real-valued.  Let us fix $\varepsilon>0.$ There exists $\lambda_0(\varepsilon)$ such that for any $\lambda> \lambda_0(\varepsilon)$ we have that
     \begin{equation}
         -\varepsilon+\smallint_S \sigma_F d\mu_{L}\leq \frac{1}{N(\lambda)}\sum_{k: \lambda_k\leq \lambda}(F\phi_k,\phi_k)\leq \varepsilon+\smallint_S \sigma_F d\mu_{L},
     \end{equation}
where we have written
$$\smallint_S \sigma_F d\mu_{L} = \smallint_{S(\Lambda \cap \Delta_{S^*X\times S^*X } )}\sigma_F d\mu_{L} $$ by denoting $S=S(\Lambda \cap \Delta_{S^*X\times S^*X } )$  to simplify the notation. At this point we recall the Weyl formula
$$ N(\lambda)=C_n\lambda^n+O(\lambda^{n-1}) .$$
In particular, if $\lambda>2\lambda_0(\varepsilon)$ we have the inequality
\begin{equation}
    \frac{1}{N(\lambda/2)}\sum_{k: \lambda_k\leq \lambda/2}(F\phi_k,\phi_k)\leq \varepsilon+\smallint_S \sigma_Fd\mu_L,
\end{equation} from which we have that 
$$ -\sum_{k: \lambda_k\leq \lambda/2}(F\phi_k,\phi_k)\geq -N(\lambda/2)\left(\varepsilon+\smallint_S \sigma_Fd\mu_L\right) $$
$$=-\left( \frac{C_n}{2^n}\lambda^n+\frac{1}{2^{n-1}}O(\frac{1}{2^{n-1}}\lambda^{n-1})\right)\left(\varepsilon+\smallint_S \sigma_Fd\mu_L\right).$$ On the other hand note that the inequality
$$ \frac{1}{N(\lambda)}\sum_{k: \lambda_k\leq \lambda}(F\phi_k,\phi_k) \geq -\varepsilon+\smallint_S \sigma_F d\mu_{L}  $$
can be written as follows
\begin{equation}
   \sum_{k: \lambda_k\leq \lambda}(F\phi_k,\phi_k)\geq \left(C_n\lambda^n+O(\lambda^{n-1})\right) \left( -\varepsilon+\smallint_S \sigma_F d\mu_{L}\right).
\end{equation} Now let us estimate from below the contribution of the terms $(T\phi_k,\phi_k)$ in the region
$$\left\{\lambda_k: {\lambda}/{2}<\lambda_k\leq \lambda\right\}.$$ Note that
\begin{align*}
    \sum_{k:\frac{\lambda}{2}<\lambda_k\leq \lambda}(F\phi_k,\phi_k) &= \sum_{k: \lambda_k\leq \lambda}(F\phi_k,\phi_k)-\sum_{k: \lambda_k\leq \lambda/2}(F\phi_k,\phi_k)\\
    &\geq \left(C_n\lambda^n+O(\lambda^{n-1})\right) \left( -\varepsilon+\smallint_S \sigma_F d\mu_{L}\right)\\
    &-\left( \frac{C_n}{2^n}\lambda^n+\frac{1}{2^{n-1}}O(\frac{1}{2^{n-1}}\lambda^{n-1})\right)\left(\varepsilon+\smallint_S \sigma_Fd\mu_L\right)\\
    &=C_n\lambda^{n}\left( -\varepsilon+\smallint_S \sigma_F d\mu_{L}\right)+O(\lambda^{n-1})\left( -\varepsilon+\smallint_S \sigma_F d\mu_{L}\right)\\
    &-\frac{C_n}{2^n}\lambda^n\left(\varepsilon+\smallint_S \sigma_Fd\mu_L\right)+O(\frac{1}{2^{n-1}}\lambda^{n-1})\left(\varepsilon+\smallint_S \sigma_Fd\mu_L\right)\\
    &=C_n\lambda^{n}\left( -\varepsilon+\smallint_S \sigma_F d\mu_{L}\right)-\frac{C_n}{2^n}\lambda^n\left(\varepsilon+\smallint_S \sigma_Fd\mu_L\right)
    +O(\lambda^{n-1})\\
    &=C_{n}\lambda^n\left(\smallint_S \sigma_F d\mu_{L}\left(1-\frac{1}{2^n}\right)-\varepsilon\left(1+\frac{1}{2^n}\right)\right)+O(\lambda^{n-1}).
\end{align*}The case of $F_1$ can be proved in a similar way. The proof is complete.     
\end{proof}
  
Now, we have the following auxiliary proposition. It is a partial characterisation of trace class Fourier integral operators. 
\begin{proposition}\label{propo} 
Let $X$ be a compact manifold of dimension $n.$  Let us consider a Fourier integral operator  $T\in  I^\mu_{1,0}(X;\Lambda) .$ Then $T:L^{2}(X)\rightarrow L^{2}(X)$ belongs to the Schatten class $S_1(L^2(X)),$  namely, $T$ is of trace class on $L^2,$ if and only if  $\mu <-n,$ provided that at least one of the following averages 
\begin{equation}\label{positivity}
    \smallint_{S(\Lambda \cap \Delta_{S^*X\times S^*X } )}\textnormal{Re}(\sigma_{T\sqrt{1-\Delta_g}^{-\mu}}) d\mu_{L}, \quad \smallint_{S(\Lambda \cap \Delta_{S^*X\times S^*X } )}\textnormal{Im}(\sigma_{T\sqrt{1-\Delta_g}^{-\mu}})  d\mu_{L},
\end{equation}
is strictly positive.    
\end{proposition}
\begin{proof}
     In view of Proposition \ref{Propo:Schatten}, the order condition $\mu<-n,$ implies that $T\in S_1(L^2(X))$ (i.e. $T$ is of trace class on $L^{2}(X)$). Now, let us consider  $T\in  I^\mu_{1,0}(X;\Lambda) $  to be of trace class. Since $$\textnormal{Re}(T)=\frac{T+T^*}{2}$$ and $$\textnormal{Im}(T)=\frac{T-T^*}{2i},$$ are also of trace class we have that 
    \begin{equation}
        \textnormal{Tr}(T)= \textnormal{Tr}(\textnormal{Re}(T))+i\textnormal{Tr}(\textnormal{Im}(T)).
    \end{equation} Let $F_0=\textnormal{Re}(T),$ and $F_1=\textnormal{Im}(T).$ Note that each $F_\ell$ is self-adjoint on $L^{2}(X),$ where $\ell=0,1.$ Let us estimate the trace $\textnormal{Tr}(F_\ell).$ Indeed,  if $(\phi_k,\lambda_k)$ are the spectral data of the operator $\sqrt{1-\Delta_g},$ 
we have that 
\begin{align*}
    \textnormal{Tr}(F_\ell)=\sum_{\lambda_k}(F_\ell\phi_k,\phi_k) &\asymp \sum_{s=0}^{\infty}\sum_{2^{s-1}<\lambda_k\leq 2^{s}  } (F_\ell\phi_k,\phi_k)\\
    &= \sum_{s=0}^{\infty}\sum_{2^{s-1}<\lambda_k\leq 2^{s}  } \lambda_k^{\mu}(F_\ell \lambda_k^{-\mu}\phi_k,\phi_k)\\
     &= \sum_{s=0}^{\infty}\sum_{2^{s-1}<\lambda_k\leq 2^{s}  } \lambda_k^{\mu}(F_\ell \sqrt{1-\Delta_g})^{-\mu}\phi_k,\phi_k),
\end{align*}which leads to the estimate
\begin{equation}\label{estimate:trace}
     \textnormal{Tr}(F_\ell)=\sum_{\lambda_k}(F_\ell\phi_k,\phi_k)\asymp \sum_{s=0}^{\infty}\sum_{2^{s-1}<\lambda_k\leq 2^{s}  } \lambda_k^{\mu}(F_\ell \sqrt{1-\Delta_g})^{-\mu}\phi_k,\phi_k).
\end{equation}
Note also that
\begin{align*}
    (\textnormal{Re}(T\sqrt{1-\Delta_g}^{-\mu})\phi_k,\phi_k  ) &=\frac{(T\sqrt{1-\Delta_g}^{-\mu}\phi_k,\phi_k)+(\sqrt{1-\Delta_g}^{-\mu}T^*\phi_k,\phi_k)}{2}\\
    &= \frac{(T\lambda_k^{-\mu}\phi_k,\phi_k)+(T^*\phi_k,\sqrt{1-\Delta_g}^{-\mu}\phi_k)}{2}\\
     &= \frac{(T\lambda_k^{-\mu}\phi_k,\phi_k)+(T^*\phi_k,\lambda_k^{-\mu}\phi_k)}{2}\\
      &= \lambda_k^{-\mu}\frac{(T\phi_k,\phi_k)+(T^*\phi_k,\phi_k)}{2}\\
      &= \lambda_k^{-\mu}(\textnormal{Re}(T)\phi_k,\phi_k)\\
      &=(\textnormal{Re}(T)\sqrt{1-\Delta_g}^{-\mu}\phi_k,\phi_k  )
\end{align*} from which we deduce that 
\begin{align*}
     (\textnormal{Re}(T\sqrt{1-\Delta_g}^{-\mu})\phi_k,\phi_k  )=(\textnormal{Re}(T)\sqrt{1-\Delta_g}^{-\mu})\phi_k,\phi_k  )=(F_0\sqrt{1-\Delta_g}^{-\mu})\phi_k,\phi_k  ).
\end{align*}
We have also that
\begin{align*}
    (\textnormal{Im}(T\sqrt{1-\Delta_g}^{-\mu})\phi_k,\phi_k  ) &=\frac{(T\sqrt{1-\Delta_g}^{-\mu}\phi_k,\phi_k)-(\sqrt{1-\Delta_g}^{-\mu}T^*\phi_k,\phi_k)}{2i}\\
    &= \frac{(T\lambda_k^{-\mu}\phi_k,\phi_k)-(T^*\phi_k,\sqrt{1-\Delta_g}^{-\mu}\phi_k)}{2i}\\
     &= \frac{(T\lambda_k^{-\mu}\phi_k,\phi_k)-(T^*\phi_k,\lambda_k^{-\mu}\phi_k)}{2i}\\
      &= \lambda_k^{-\mu}\frac{(T\phi_k,\phi_k)-(T^*\phi_k,\phi_k)}{2i}\\
      &= \lambda_k^{-\mu}(\textnormal{Im}(T)\phi_k,\phi_k)\\
      &=(\textnormal{Im}(T)\sqrt{1-\Delta_g}^{-\mu}\phi_k,\phi_k  ),
\end{align*} from which we deduce that 
\begin{align*}
     (\textnormal{Im}(T\sqrt{1-\Delta_g}^{-\mu})\phi_k,\phi_k  )=(\textnormal{Im}(T)\sqrt{1-\Delta_g}^{-\mu})\phi_k,\phi_k  )=(F_1\sqrt{1-\Delta_g}^{-\mu})\phi_k,\phi_k  ).
\end{align*}
Summarising the analysis above we have obtained the following equalities
\begin{equation}\label{Re:identity}
     (\textnormal{Re}(T\sqrt{1-\Delta_g}^{-\mu})\phi_k,\phi_k  )=(F_0\sqrt{1-\Delta_g}^{-\mu})\phi_k,\phi_k  ),
\end{equation} and
\begin{equation}\label{Im:identity}
     (\textnormal{Im}(T\sqrt{1-\Delta_g}^{-\mu})\phi_k,\phi_k  )=(F_1\sqrt{1-\Delta_g}^{-\mu})\phi_k,\phi_k  ).
\end{equation}
Remember that, from Lemma \ref{dyadic:lemma} applied to $S= T\sqrt{1-\Delta_g}^{-\mu},$ there exists $\lambda_0(\varepsilon)$ such that for $\lambda\geq \lambda_0(\varepsilon), $ we have the estimate from below
$$\sum_{\frac{\lambda}{2}<\lambda_k\leq \lambda}(\textnormal{Re}(T\sqrt{1-\Delta_g}^{-\mu})\phi_k,\phi_k)$$ 
\begin{equation}
  \geq C_n\lambda^n\left(\left(1-\frac{1}{2^n}\right)\smallint_{S(\Lambda \cap \Delta_{S^*X\times S^*X } )}\textnormal{Re}(\sigma_{S}) d\mu_{L}-\varepsilon\left(1+\frac{1}{2^n}\right) \right)+O(\lambda^{n-1}).
  \end{equation}
Without loss of generality we can assume that 
$$\smallint_{S(\Lambda \cap \Delta_{S^*X\times S^*X } )}\textnormal{Re}(\sigma_S) d\mu_{L}>0.$$
We can choose $\varepsilon>0$ such that 
$$ C_\varepsilon=\left(1-\frac{1}{2^n}\right)\smallint_{S}\textnormal{Re}(\sigma_S) d\mu_{L}-\varepsilon\left(1+\frac{1}{2^n}\right)>0.$$ Indeed, such a $\varepsilon$ satisfies the inequality
$$0<\varepsilon< \frac{\left(1-\frac{1}{2^n}\right)}{\left(1+\frac{1}{2^n}\right)}\smallint_{S}\textnormal{Re}(\sigma_S).$$
 In consequence we have that
 \begin{equation}\label{estimate:from:below}
     \sum_{2^{s-1}<\lambda_k\leq 2^{s}}(\textnormal{Re}(T\sqrt{1-\Delta_g}^{-\mu})\phi_k,\phi_k)\geq  C_nC_\varepsilon 2^{sn}+O(2^{s(n-1)})
 \end{equation}
where $C_n$ is a geometric constant and $C_\varepsilon>0$. Note that, in view of \eqref{Re:identity},  we have the identity $$ 
     (\textnormal{Re}(T\sqrt{1-\Delta_g}^{-\mu})\phi_k,\phi_k  )=(F_0\sqrt{1-\Delta_g}^{-\mu})\phi_k,\phi_k  ).$$  Using this fact as well as the equality in \eqref{estimate:trace}  we can estimate the trace of $F_0$ as follows,
\begin{align*}
     \textnormal{Tr}(F_0) &\asymp \sum_{s=0}^{\infty}\sum_{2^{s-1}<\lambda_k\leq 2^{s}  } \lambda_k^{\mu}(F_0 \sqrt{1-\Delta_g})^{-\mu}\phi_k,\phi_k)\\
     &\asymp \sum_{s=0}^{\infty}\sum_{2^{s-1}<\lambda_k\leq 2^{s}  } 2^{s\mu}(T_0 \sqrt{1-\Delta_g})^{-\mu}\phi_k,\phi_k)\\
     &= \sum_{s=0}^{\infty}\sum_{2^{s-1}<\lambda_k\leq 2^{s}  } 2^{s\mu}(\textnormal{Re}(T \sqrt{1-\Delta_g})^{-\mu})\phi_k,\phi_k)\\
     &\gg_{\varepsilon} \sum_{s=0}^{\infty} (C_nC_\varepsilon 2^{s(n+\mu)}+O(2^{s(\mu+n-1)})),
\end{align*} where in the last line we have used \eqref{estimate:from:below}. If the order $\mu$ is such that  $\mu\geq -n,$ the  series $$ \sum_{s=0}^{\infty} (C_nC_\varepsilon 2^{s(n+\mu)}+O(2^{s(\mu+n-1)})),$$ diverges which contradicts that $T_0$ is of trace class. In consequence, we have that $\mu<-n.$ The proof is complete.  
\end{proof}

\begin{theorem}\label{th:trace:class} Let $X$ be a compact manifold of dimension $n.$  Let us consider a Fourier integral operator  $T\in  I^\mu_{1,0}(X;\Lambda) .$ Then $T:L^{2}(X)\rightarrow L^{2}(X)$ belongs to the Schatten class $S_1(L^2(X)),$  namely, $T$ is of trace class on $L^2,$ if and only if  $\mu <-n.$ 
    
\end{theorem}
\begin{proof} In view of Proposition \ref{Propo:Schatten}, the order condition $\mu<-n,$ implies that $T\in S_1(L^2(X))$ (i.e. $T$ is of trace class on $L^{2}(X)$).  On the other hand, assume that $T$ having order $\mu$ is of trace class. Without loss of generality we assume that $T\neq 0.$ The idea is to find a pseudo-differential operator $F$ of order $\mu,$ of trace class, since applying Theorem \ref{th:trace:class:PSDO}  we could prove that $\mu<-n.$ To do this we define
$$F=T^*T \sqrt{1-\Delta_g}^{-\mu}.$$ Note that $T^*T\in \Psi^{2\mu}_{1,0}(X)$ is  a pseudo-differential operator of order $2\mu,$ and then $T^*T \sqrt{1-\Delta_g}^{-\mu}$ is a pseudo-differential operator of order $\mu.$ Since $T$ is of trace class, we also have that $T^*$ is of class trace. Since $T \sqrt{1-\Delta_g}^{-\mu}$ is a Fourier integral operator of order zero, and then it is a bounded operator on $L^2(X),$ we have that  $F=T^*T \sqrt{1-\Delta_g}^{-\mu}$ is a trace class pseudo-differential operator of order $\mu$. Then, the inequality $\mu<-n$ follows from  Theorem \ref{th:trace:class:PSDO}.  The proof of Theorem \ref{th:trace:class} is complete. However, let us give an alternative proof for the theorem based in Proposition \ref{propo}. For this we consider again the operator $F=T^*T \sqrt{1-\Delta_g}^{-\mu}$ which is a trace class pseudo-differential operator of order $\mu.$ The idea is to prove that
$$  \smallint_{S(\Lambda \cap \Delta_{S^*X\times S^*X } )}\textnormal{Re}(\sigma_{F\sqrt{1-\Delta_g}^{-\mu}}) d\mu_{L}>0.$$ Note that
$$F\sqrt{1-\Delta_g}^{-\mu}=T^*T \sqrt{1-\Delta_g}^{-2\mu}.$$ The principal symbol of $T^*T$ is given by 
\begin{equation}
   \vert \sigma_T(x,y,\theta)\vert^2\vert \textnormal{Det}   (D\Phi)\vert,
\end{equation}where $\Phi$ is the phase function of $T,$ see e.g. Zelditch \cite[Page 145]{Zelditch}. In consequence the principal symbol of $F\sqrt{1-\Delta_g}^{-\mu}$ is given by
$$ \vert \sigma_T(x,y,\theta)\vert^2\vert\textnormal{Det}   (D\Phi)\vert\Vert\theta \Vert^{-\mu},$$ when the co-variable $\theta$ is such that $\theta\neq 0.$ In consequence, we have proved that
$$  \smallint_{S(\Lambda \cap \Delta_{S^*X\times S^*X } )}\textnormal{Re}(\sigma_{F\sqrt{1-\Delta_g}^{-\mu}}) d\mu_{L}=\smallint_{S(\Lambda \cap \Delta_{S^*X\times S^*X } )}\vert \sigma_T(x,y,\theta)\vert^2\vert D(\Phi)\vert\Vert\theta \Vert^{-\mu} d\mu_{L}>0,$$ where we have used that $\sigma_T\not\equiv 0,$ because $T\neq 0.$
The proof is complete.
\end{proof}
For any $r>1,$ we provide the following sufficient condition on the order of a Fourier integral operator in order to guarantee its membership in the Schatten class of order $r>0.$

\begin{lemma}\label{Invariance:trace:2} Let $X$ be a compact manifold of dimension $n.$  Let us consider a Fourier integral operator  $A\in  I^\mu_{1,0}(X;\Lambda) $ and let $r>1.$  If $A\in S_r(L^2(X))$ then $m\leq -n/r.$
\end{lemma}
\begin{proof}
    Let $\varepsilon>0.$ Consider the operator $\sqrt{1-\Delta}^s$ where $s$ satisfies
$$ s:=-\frac{n}{q}-\varepsilon<-\frac{n}{q},$$ and where $$q:=r/(r-1)$$ is the conjugate exponent of $r.$ Then, $q$ is given by the identity 
$1=\frac{1}{q}+\frac{1}{r}.$ Note that since $r>1,$ one has that $q>1.$ Note also that $$\sqrt{1-\Delta}^s\in S_q(L^2(X)).$$ Since $q$ and $r$ and conjugate exponents, we have that
\begin{equation*}
    \tilde{A}:=A\sqrt{1-\Delta}^s\in S_1(L^2(X)).
\end{equation*}By Theorem \ref{th:trace:class} we have that the order $\mu+s$ of $\tilde A\in S_1$ satisfies the inequality 
$\mu+s<-n.$ Then, we have that 
$$  \mu<-s-n=\frac{n}{q}+\varepsilon -n=n\left(\frac{1}{q}-1\right)+\varepsilon =-\frac{n}{r}+\varepsilon\,, $$
where the last hold true for any $\epsilon>0$.
Taking $\varepsilon\rightarrow 0^+$ we have that $\mu\leq -n/r.$
\end{proof}

\begin{theorem}\label{even:condition} Let $X$ be a compact manifold of dimension $n.$  Let us consider a Fourier integral operator  $T\in  I^\mu_{1,0}(X;\Lambda) $ and let $r\in \mathbb{N}= (0,\infty)\cap \mathbb{Z}.$ Then $T:L^{2}(X)\rightarrow L^{2}(X)$ belongs to the Schatten class $S_r(L^2(X)),$  if and only if  $\mu <-n/r.$ 
    
\end{theorem}
\begin{proof} In view of Proposition \ref{Propo:Schatten}, the order condition  $\mu<-n/r,$ $0<r<\infty,$ implies that $T\in S_r(L^2(X)).$ Now, let us prove the converse statement, that is if $T\in S_r(L^2(X))$, then $m<-n/r.$  Note that the case $r=1$ has been proved in \ref{th:trace:class}. Now, we assume that $r> 1$ and that $r\in \mathbb{Z}.$ Note that $$F=T^*T \sqrt{1-\Delta_g}^{-\mu},$$ is a pseudo-differential operator of order $\mu.$ Note also that $T \sqrt{1-\Delta_g}^{-\mu}$ is a Fourier integral operator of order zero and then it is bounded on $L^2(X).$ Also, note that  $T^*\in S_r(L^2(X)),$ and in consequence  $F\in S_r(L^2(X)).$ To prove the theorem let us show that the borderline $\mu=-n/r$ in Lemma \ref{Invariance:trace:2} is not possible. 
Observe that
\begin{equation}
    F^r= F\circ F\circ\cdots F\in S_{r}(L^2(X))\circ\cdots\circ S_{r}(L^2(X))\subseteq S_{1}(L^2(X)),
\end{equation} where the composition $\circ$ is taken $r$-times. However, the order of the trace class operator $ F^r$ is $$\mu r=-n$$ and this contradicts the conclusion in Theorem  \ref{th:trace:class:PSDO}. So, necessarily $\mu<-n/r.$     
\end{proof}
\subsection{Distribution of eigenvalues of Fourier integral operators on $L^2$}\label{FIO:dist}  Next we analyse the distribution of eigenvalues of Fourier integral operators on $L^2(X).$ We record the following result due to Hermann Weyl (see \cite{Weyl1949}): if $T:H\rightarrow H$ is a bounded operator on a Hilbert space $H,$ and it belong to the Schatten class $S_r(H),$ then the eigenvalues $\{\lambda_j(T)\}$ of $T,$ and its singular values $\{s_j(T)\}$ satisfy the inequality
\begin{equation}\label{Weyl:ineq}
    \sum_{j=0}^{\infty}|\lambda_j(T)|^r\leq  \sum_{j=0}^{\infty}s_j(T)^r.
\end{equation}
{Weyl's inequality can also be presented in the following way: if $T:H\rightarrow H$ is a compact linear operator, then for any $\ell\in \mathbb{N}$ and all $0<p<\infty,$ one has the inequality $$ \sum_{j=0}^{\ell}|\lambda_j(T)|^p\leq  \sum_{j=0}^{\ell}s_j(T)^p.$$}
We have the following result.
\begin{theorem}\label{th:trace:class:dist} Let $X$ be a compact manifold of dimension $n,$ and let $0<r<\infty.$  Let us consider a Fourier integral operator  $T\in  I^\mu_{1,0}(X;\Lambda) $ of order $\mu<-n/r.$ Then the system of eigenvalues of $T,$ satisfies the rate of decay
\begin{equation}\label{j:r}
    |\lambda_j(T)|=O(j^{-\frac{1}{r}}).
\end{equation}    
\end{theorem}
\begin{proof}
    In view of the order condition $\mu<-n/r,$ we have that $T$ belongs to the Schatten class $S_r(L^2(X)),$ see Proposition \ref{Propo:Schatten}.  Then, if $\{s_j(T)\}$ denotes the system of singular values of $T,$ we have that $\sum_{j=0}^{\infty}s_j(T)^r<\infty.$ In consequence, using Weyl's inequality \eqref{Weyl:ineq}, we have that 
    the system of eigenvalues $\{\lambda_j(T)\}$ of $T$ is $r$-summable, that is $\sum_{j=0}^{\infty}|\lambda_j(T)|^r<\infty.$ In consequence we have that 
    $$\lim_{j\rightarrow\infty}j|\lambda_j(T)|^r=0,$$ from which we deduce \eqref{j:r}. The proof is complete. 
\end{proof}

\bibliographystyle{amsplain}

\end{document}